\documentclass [11pt]{amsart}

\usepackage{hyperref}
\usepackage{cite}
\usepackage{amssymb,bm}
\usepackage{amsmath}
\usepackage{amsthm}
\usepackage{amsfonts,dsfont}
\usepackage{array}
\usepackage{arydshln}
\usepackage{bbm}
\usepackage{bm}
\usepackage[makeroom]{cancel}
\usepackage{graphicx}
\usepackage[rgb,usenames,dvipsnames,svgnames,table]{xcolor}
\usepackage[toc,page]{appendix}
\usepackage{layout}
\usepackage[normalem]{ulem}
\usepackage{mathtools}
\usepackage{multicol}
\usepackage{mathrsfs}
\usepackage{comment}
\usepackage{enumerate,enumitem}
\usepackage{float}
\usepackage{circuitikz}

\newtheorem{theorem}{Theorem}[section]

\newtheorem{assumption}[theorem]{Assumption}

\newtheorem{proposition}[theorem]{Proposition}

\newtheorem{corollary}[theorem]{Corollary}
\newtheorem{lemma}[theorem]{Lemma}
\newtheorem{remark}[theorem]{Remark}
\newtheorem{example}[theorem]{Example}
\newtheorem{examples}[theorem]{Examples}
\newtheorem{foo}[theorem]{Remarks}

\newtheorem{open question}[theorem]{Open Question}
\newtheorem{c/p}[theorem]{Conjecture/Proposition}

\newcommand{\norm}[1]{{\left\lVert{#1}\right\rVert}}

\newcommand{\abs}[1]{\left|#1\right|} 

\def\vint{\mathop{\mathchoice%
		{\setbox0\hbox{$\displaystyle\intop$}\kern 0.22\wd0%
			\vcenter{\hrule width 0.6\wd0}\kern -0.82\wd0}%
		{\setbox0\hbox{$\textstyle\intop$}\kern 0.2\wd0%
			\vcenter{\hrule width 0.6\wd0}\kern -0.8\wd0}%
		{\setbox0\hbox{$\scriptstyle\intop$}\kern 0.2\wd0%
			\vcenter{\hrule width 0.6\wd0}\kern -0.8\wd0}%
		{\setbox0\hbox{$\scriptscriptstyle\intop$}\kern 0.2\wd0%
			\vcenter{\hrule width 0.6\wd0}\kern -0.8\wd0}}%
	\mathopen{}\int}

\title{Chaos processes as rough paths}

\author{Guang Yang}
\address{Department of Mathematics\\
	Purdue University\\
	West Lafayette, IN 47907, U.S.A. 
}
\email{yang2220@purdue.edu}

\begin{document}
\begin{abstract}
	In this article we investigate the rough paths structure of a process $X_t$ living in a fixed Wiener chaos. Specifically, we formulate various types of rough lifts of $X_t$ and study their properties. As application, we study the integrabilities of quantities related to rough differential equations driven by $X_t$. 
\end{abstract}

\maketitle
\tableofcontents

\section{Introduction}
Since the seminal paper by T. Lyons \cite{MR1654527}, rough path theory has been developed rapidly in the past two decades. Among the most prominent directions, is the study of Gaussian rough paths (see \cite{MR1883719} \cite{MR2667703} ), which has led to a wide variety of interesting results. As a natural counterpart to this development, significant progresses have been achieved in the study of rough differential equations (RDEs) driven by Gaussian rough paths. Such differential equations can be written in the form
\begin{equation}
	\label{RDE in introduction}
	Y_t=y_0+\sum_{i=1}^{d}\int_{0}^{t}V_i(Y_s)d\boldsymbol{X}^i_t+\int_{0}^{t}V_0(Y_s)ds,\ y_0\in\mathbb{R}^d,\ t\in[0,1],
\end{equation}
where $\{V_i\}_{0\leq i\leq d}$ are vector fields on $\mathbb{R}^d$ with suitable regularity assumption and $\boldsymbol{X}_t$ is a Gaussian rough path.
 
Many of the rich structures associated with Gaussian processes fit in the study of such RDEs amazingly well. Interactions with Malliavin calculus had led to a thorough study of the law of the solution $Y_t$. In this direction, the existence and smoothness of the density of $Y_t$ were studied in \cite{MR2680405} and \cite{MR3298472}. Upper bound and strict positivity of the density and hitting probabilities were developed in \cite{MR3531675}. Regularity and ergodicity were studied in \cite{ MR3112925}. Adapting the isoperimetric inequality of Gaussian measures to the rough paths setting allows one to study the integrabilities of the Jacobian process $J^Y_t$ of $Y_t$ \cite{MR3112937}, which, in turn, enables the study of integrabilities of Malliavin derivatives of $Y_t$ \cite{MR3229800}. 

However, the study of the Malliavin derivatives of $Y_t$ in \cite{MR3229800} appears to be far more difficult than one would have expected. This is due to the fact that, with a Gaussian rough path, one can only write the ``directional derivatives'' of the form $\langle D^kY_t, h_1\otimes\cdots h_k \rangle_{\mathcal{H}^{\otimes k}}$ in a consistent way, which eventually leads only to the operator norm of $D^kY_t$. To get the the tensor norm (or equivalently, Hilbert-Schmidt norm) of $D^kY_t$, Inahama had to devise a rather technical method, whose starting point is the unique feature that the Malliavin derivative of a Gaussian random variable is deterministic.  
 
We will present this new idea that instead of a rough path above the process $X_t$ along, we can construct a rough path above the enlarged process $\hat{X}_t=(X_t, DX_t)$. We shall see later that when $X_t$ is Gaussian, this enlargement does not cause any extra difficulty for all the previous studies. The advantage of such endeavour is immediately apparent. With a rough path above $\hat{X}_t$, all Malliavin derivatives of $Y_t$ can be written as solutions to RDEs driven by $\hat{X}_t$ along linear vector fields (see section 5 below). Thus, the study of their integrabilities can be combined into a systematic study of linear RDEs driven by $\hat{X}_t$.  

This same development can be naturally carried our for processes whose Malliavin derivatives vanish above certain level, where we need to consider $\hat{X}_t=(X_t, DX_t, \cdots, D^nX_t)$. This leads to the so called chaos processes, which generalize Gaussian processes while still inherit many useful structures. The study of chaos random variables have gained increasing attentions in recent years (see, for instance, \cite{MR3003367} \cite{MR3035750} \cite{MR3498012}  \cite{MR4020062}). Our study shall focus on chaos processes. 

More specifically, we will fix $n\geq 1$ and consider a process $X_t\in \mathbb{R}^d$ parametrized on the compact interval $[0,1]$, whose components are independent copies of $I_n(f_t)$, where $I_n$ is the $n$-th multiple Wiener integral and $f_t\in \mathcal{H}^{\otimes n}$, where $\mathcal{H}$ is the Cameron-Martin space of the underlying Wiener space. Note that when $n=1$, $X_t$ is Gaussian.

Our study mainly consists of two parts. First, we will construct rough lifts for $X_t$ and $\hat{X}_t$. Several properties of the rough lifts will be studied, including a large deviation principle. Then, as an application, we study the integrability of Malliavin derivatives and Jacobian process of solutions to RDEs of form \eqref{RDE in introduction} driven by $X_t$.

There are two main assumptions. The first one is the same assumption we used for Gaussian processes.
\begin{assumption}
	\label{Assumption 1}
	Let $X_t=I_n(f_t)$ be a real-valued chaos process. For $0\leq s,t\leq 1$, define $R(s,t)=R_X(s,t):=\mathbb{E}(X_sX_t)$. We assume that $R(s,t)$ has finite $\rho$-variation on $[0,1]\times [0,1]$ for some $\rho\in[1,3/2)$. Moreover, we assume the $\rho$-variation of $R(s,t)$ is H\"older-controlled in the sense that
	\[ \norm{R}^\rho_{\rho-var;[s,t]\times [s,t]}\leq \abs{t-s}. \]	
\end{assumption}
It is worth pointing out that the H\"older controlled property is not restrictive as one can always reparameterize the process $X_t$ (see \cite{MR3298472}).

A simple application of the moment equivalence of Wiener chaos gives that for any $p>1$,
\[\mathbb{E}\abs{X_t-X_s}^p\leq (p-1)^{n/2}\left(\mathbb{E}\abs{X_t-X_s}^2\right)^{\frac{p}{2}}\leq (p-1)^{n/2} \abs{t-s}^{\frac{p}{2\rho}}.  \]
By Kolmogorov continuity theorem we see that assumption \ref{Assumption 1} implies $X_t$ is $1/(2\rho')$-H\"older continuous (or has finite $2\rho'$-variation) for any $\rho'>\rho$ almost surely. As $\rho\in [1, 3/2)$, we only need to construct the second order iterated integral for $X_t$ to get a rough lift. 

It is hardly a surprise that an assumption on covariance along is not sufficient for all our purposes. We shall encounter many difficulties that did not appear for Gaussian rough paths. For instance, when we apply translation in the direction of $\mathcal{H}$, one has the Malliavin-Stroock formula 
\begin{equation*}
	X_t(\omega+rh)=X_t(\omega)+\sum_{k=1}^{n}\frac{r^k}{k!}\cdot\langle D^kX_t(\omega), h^{\otimes k} \rangle_{\mathcal{H}^{\otimes k}},\ r>0,\ h\in\mathcal{H}.
\end{equation*}
To get the rough lift of $ X_t(\omega+rh)$, we must define integrals of the form
\begin{equation}
		\label{Cross integraltion in introduction}
		\int_{s}^{t}D^mX_{s,r}\otimes dD^kX_r,\ 0\leq m,k,\leq n,
\end{equation}
where $X_{s,r}=X_r-X_s$. This observation is another motivation for us to consider the rough lift of $\hat{X}_t$ instead of $X_t$, since terms in \eqref{Cross integraltion in introduction} are precisely the second level iterated integrals of $\hat{X}_t$. It turns out that one needs a condition genuinely stronger than assumption \ref{Assumption 2} in order to handle terms in \eqref{Cross integraltion in introduction}. We will come back to this point with more details in section 3 below.

The following assumption, enhanced from assumption \ref{Assumption 1} by adding what emerges as the condition needed to deal with terms in \eqref{Cross integraltion in introduction}, will be our second assumption that replaces assumption \ref{Assumption 1} when needed.

Recall for $a,b\in \mathcal{H}^{\otimes n}$ that are symmetric, the symmetric contraction of degree $r$ is defined to be $a\hat{\otimes}_rb=\widehat{\langle a,b \rangle}_{\mathcal{H}^{\otimes r}}\in\mathcal{H}^{\otimes (2n-2r)}$ for $1\leq r\leq n$, where $\widehat{\langle a,b \rangle}_{\mathcal{H}^{\otimes r}}$ is the symmetrization of $\langle a,b \rangle_{\mathcal{H}^{\otimes r}}$. More details about contractions are given in section 2.
\begin{assumption}
	\label{Assumption 2}
	Let $X_t=I_n(f_t)$ be a real-valued chaos process. We assume that there exists a control $\omega:[0,1]^2\rightarrow [0,\infty) $, such that for some $\rho\in[1,3/2)$,  $1\leq r \leq n$ and any $[s,t],[u,v]\subset [0,1]$, we have
	\begin{equation}
		\label{Assumption 2 main equation}
		\norm{f_{s,t}\hat{\otimes}_{r}f_{u,v}}_{\mathcal{H}^{\otimes (2n-2r)}}\leq \omega^{\frac{1}{\rho}}([s,t]\times [u,v]).
	\end{equation}
	Moreover, we assume $\omega$ satisfies
	\[ \omega([s,t]\times [s,t])\leq \abs{t-s}. \]	
\end{assumption}
\begin{remark}
	The contractions $f_{s,t}\hat{\otimes}_{r}f_{u,v}$ appear in \eqref{Assumption 2 main equation} are frequently used to control the product of chaos random variables  and have naturally appeared in many studies of chaos random variables (see section 2 below).
\end{remark}
\begin{remark}
	We can give assumption \ref{Assumption 2} a more intuitive meaning. If we further assume a non-degenerate lower bound  $0<C\leq\norm{f_{r}}_{\mathcal{H}^{\otimes n}}$ for all $ r\in [s,t]$, then assumption \eqref{Assumption 2}  implies 
	\[ d_W(X_{r}, N)\leq C'\abs{r}^{\frac{1}{\rho}},\ \forall r\in [s,t],    \]
	where $d_W$ is the Wasserstein distance and $N$ denotes a standard normal distribution. From this perspective, assumption \ref{Assumption 2} is essentially a regularity assumption on any interval where $X_t$ is non-degenerate. We point out that the above inequality is still valid with Kolmogorov distance and total variation distance. We refer to chapter 5 of \cite{MR2962301} for more details.  
\end{remark}
\begin{remark}
	When $r=n$, inequality \eqref{Assumption 2 main equation} becomes
	\[ \abs{\mathbb{E}(X_{s,t}X_{u,v})}=n!\abs{\langle f_{s,t}, f_{u,v} \rangle_{\mathcal{H}^{\otimes n}} }=n!\abs{f_{s,t}\hat{\otimes}_nf_{u,v}}\leq n!\cdot\omega^{\frac{1}{\rho}}([s,t]\times [u,v]),  \]
	which is nothing but assumption \ref{Assumption 1}. In particular, when $n=1$, assumption \ref{Assumption 1} and assumption \ref{Assumption 2} are identical. Hence, assumption \ref{Assumption 2} can be regraded as a natural generalization of assumption \ref{Assumption 1} to non-Gaussian Wiener chaos settings.
\end{remark}
	In section 3, we will give an easy method to construct processes that satisfy assumption \ref{Assumption 2}.

The rest of the paper is organized as follows. In section 2, we gather some preliminary materials. Section 3 is devoted to the construction of rough lifts of $X_t$ and $\hat{X}_t$. Different types of approximations are also discussed. Section 4 develops a large deviation principle for $X_t$ and its rough lift, while section 5 studies the integrability of the Malliavin derivatives and Jacobian process of solutions to RDEs driven by $X_t$. 

\textbf{Notations}: Throughout this paper, we use $\abs{\cdot}$ for Euclidean norms. Norms in Banach spaces will generally be denoted as $\norm{\cdot}$ if there is no risk of ambiguity. For $\alpha\in(0,1)$, we will write the $\alpha$-H\"older constant for a function $F_t$ as
\[ \norm{F}_{\alpha-\text{H\"ol}}:= \sup_{0\leq s<t\leq 1}\frac{\norm{F_t-F_s}}{\abs{t-s}^{\alpha}}. \] 
We will also use
\[ \norm{f_t}_{\infty}:=\sup_{t\in[0,1]}\norm{f_t}.  \]
Difference $X_t-X_s$ will be denoted by $X_{s,t}$. For estimates of the form $\abs{x}\leq C\abs{y}$;
if the constant $C$ does not depend on $x$ or $y$, we will omit it and simply write $\abs{x}\preceq\abs{y}$.

\section{Preliminaries}
\subsection{Wiener chaos}
Let $\mathcal{H}$ be a real separable Hilbert space. We say $W=\{W(h):\ h\in\mathcal{H} \}$ is an isonormal Gaussian process over $\mathcal{H}$, if $W$ is a family of centered Gaussian random variables defined on some complete probability space $(\Omega, \mathcal{F}, \mathbb{P})$ such that
\begin{equation*}
	\mathbb{E}(W(h)W(g))=\langle h,g \rangle_{\mathcal{H}}.
\end{equation*}
We will further assume that $\mathcal{F}$ is generated by $W$.

For every $k\geq 1$, we denote by $\mathcal{H}_k$ the $k$-th homogeneous Wiener chaos of $W$ defined as the closed subspace of $L^2(\Omega)$ generated by the family of random variables $\{H_k(W(h)):\ h\in\mathcal{H} \}$ where $H_k$ is the $k$-th Hermite polynomial given by
\begin{equation*}
	H_k(x)=(-1)^ke^{\frac{x^2}{2}}\frac{d^k}{dx^k}\left(e^{-\frac{x^2}{2}} \right).
\end{equation*} 
$H_0$ is by convention defined to be $\mathbb{R}$.

For any $k\geq 1$, we denote by $\mathcal{H}^{\otimes k}$ the $k$-th tensor product of $\mathcal{H}$. If $\phi_1, \phi_2,\cdots ,\phi_n\in\mathcal{H}$, we define the symmetrization of $\phi_1\otimes \cdots \otimes \phi_n$ by
\[\phi_1\hat{\otimes}\cdots \hat{\otimes} \phi_n=\frac{1}{n!}\sum_{\sigma\in\Sigma_n}\phi_{\sigma(1)}\otimes \cdots \otimes \phi_{\sigma(n)},\]
where $\Sigma_n$ is the symmetric group of $\{1,2,\cdots,n\}$. The symmetrization of $\mathcal{H}^{\otimes k}$ is denoted by $\mathcal{H}^{\hat{\otimes} k}$. We consider $f\in\mathcal{H}^{\hat{\otimes} n}$ of the form
\[f=e_{j_1}^{\hat{\otimes} k_1}\hat{\otimes} e_{j_2}^{\hat{\otimes} k_2} \hat{\otimes} \cdots \hat{\otimes} e_{j_m}^{\hat{\otimes} k_{m}},  \]
where $\{e_i\}_{i\geq 1}$ is an orthonormal basis of $\mathcal{H}$ and $k_1+\cdots +k_m=n$. The multiple Wiener-It\^o integral of $f$ is defined as
\[I_n(f)=H_{k_1}(W(e_{j_1})) \cdots H_{k_m}(W(e_{j_m})). \]
If $f,g\in \mathcal{H}^{\hat{\otimes} n}$ are of the above form, we have the following isometry 
\[\mathbb{E}(I_n(f)I_n(g))=n!\langle f,g \rangle_{\mathcal{H}^{\otimes n}}. \]
For general elements in $ \mathcal{H}^{\hat{\otimes} n}$, the multiple Wiener-It\^o integrals are defined by $L^2$ convergence with the previous isometry equality.

Let $\mathcal{G}$ be the $\sigma$-algebra generated by $\{W(h), h\in\mathcal{H}\}$, then any random variable $F\in L^2(\Omega, \mathcal{G}, \mathbb{P})$ admits an orthonormal decomposition (Wiener chaos decomposition) of the form
\begin{equation*}
	F=\sum_{k=0}^{\infty}I_k(f_k),
\end{equation*}
where $f_0=\mathbb{E}(F)$ and $f_k\in\mathcal{H}^{\hat{\otimes}k}$ are uniquely determined by $F$.

One of the most appealing properties of chaos random variables is the moment equivalence, which we would frequently employ.
\begin{proposition}
	\label{Moment equivalence}
	Let $f\in\mathcal{H}^{\hat{\otimes}n}$ and $X=I_n(f)$. Then, for any $1<p,q<+\infty$, we have
	\[\norm{X}_{L^p(\Omega)}\leq \norm{X}_{L^q(\Omega)}\leq \left(\frac{q-1}{p-1} \right)^{n/2}\norm{X}_{L^q(\Omega)}.\]
\end{proposition}

Suppose that $f\in\mathcal{H}^{\hat{\otimes} n}$ and $g\in\mathcal{H}^{\hat{\otimes} m}$, we define the symmetric contractions of $f$ and $ g$ as
\[ f\hat{\otimes}_rg=\widehat{\langle f,g \rangle}_{\mathcal{H}^{\otimes r}},\ 1\leq r\leq n\wedge m,  \]
where $\widehat{\langle f,g \rangle}_{\mathcal{H}^{\otimes r}}$ means the symmetrization of $\langle f,g \rangle_{\mathcal{H}^{\otimes r}}$. Note that since $f,g$ are symmetric, we do not have to specify with respect to which $r$ components are the scalar products taken. The product formula of chaos random variables can be written in terms of these contractions
\begin{equation*}
	I_n(f)\cdot I_m(g)=\sum_{r=0}^{m\wedge n}r\binom{n}{r}\binom{m}{r}I_{n+m-2r}(f\hat{\otimes}_rg).
\end{equation*}

\subsection{Malliavin calculus}
Let $\mathcal{H}$ and $W$ be as previous subsection. Let $\mathcal{FC}^\infty$ denote the set of cylindrical random variables of the form
\[  F=f(W(h_1), \cdots ,W(h_n )), \]
where $n\geq 1$, $h_i\in \mathcal{H}$ and $f\in C^{\infty}_b(\mathbb{R}^n)$; that means $f$ is a smooth function on $\mathbb{R}^n$ bounded with all derivatives. The Malliavin derivative of $F$ is a $\mathcal{H}$-valued random variable defined as
\begin{equation*}
	DF=\sum_{i=1}^{n}\frac{\partial f}{\partial x_i}(W(h_1), \cdots ,W(h_n ))h_i.
\end{equation*}
One can relate $DF$ with a directional derivative as
\[ \lim_{r\rightarrow \infty}\frac{F(\omega+rh)-F(\omega)}{r}=\langle DF, h \rangle_{\mathcal{H}},\ \forall h\in\mathcal{H}.  \]
By iteration, one can define the $k$-th Malliavin derivative of $F$ as a $\mathcal{H}^{\otimes k}$-valued random variable. For $m,p\geq 1$, we denote $\mathbb{D}^{m,p}$ the closure of $\mathcal{FC}^\infty$ with respect to the norm
\[ \norm{F}^p_{m,p}=\mathbb{E}(\abs{F}^p)+\sum_{k=1}^{m}\mathbb{E}\left(\norm{D^kF }^p_{\mathcal{H}^{\otimes k}}  \right).\]
We define $\mathbb{D}^{\infty}=\cap_{m,p\geq 1}\mathbb{D}^{m,p}$.

\subsection{Rough path and rough differential equation}
Let $\mathscr{X}$ be a Banach space. For $p\in [2,3)$, we define the space of $p$ rough path on $\mathscr{X}$ parameterized on the interval $[0,1]$, in symbols $\mathscr{C}^{p-var}([0,1], \mathscr{X})$, as those pairs $\boldsymbol{X}=:(X, \mathbb{X}) \in C([0,1], \mathscr{X}\oplus \mathscr{X}^{\otimes 2})$ such that
\begin{align*}
	\norm{X}_{p-var;[0,1]}&=\left(\sup_{(s,t)\in D([0,1])}\norm{X_{s,t}}^p \right)^{\frac{1}{p}}<+\infty,\\
	\norm{\mathbb{X}}_{p/2-var;[0,1]}&=\left(\sup_{(s,t)\in D([0,1])}\norm{\mathbb{X}_{s,t}}^{p/2} \right)^{\frac{2}{p}}<+\infty,
\end{align*}
where the supreme is taken over all partitions of $[0,1]$. Moreover, for $0\leq s\leq u \leq t \leq 1$, we have the following Chen's relation 
\begin{equation*}
	\mathbb{X}_{s,t}=\mathbb{X}_{s,u}+\mathbb{X}_{u,t}+X_{s,u}\otimes X_{u,t}.
\end{equation*}
We will say $\boldsymbol{X}$ is a rough lift of $X$. One can equip $\mathscr{C}^{p-var}([0,1], \mathscr{X})$ with the homogeneous $p$-variation norm
\begin{equation*}
	\norm{\boldsymbol{X}}^p_{p-var;[0,1]}=\norm{X}^p_{p-var;[0,1]}+\norm{\mathbb{X}}^{p/2}_{p/2-var;[0,1]}.
\end{equation*}

At its most fundamental, we need to construct $\mathbb{X}_t$ from a given $X_t$ in order to get a rough lift, which in general is not unique. In the special case that $X\in C^{1-var}([0,1],\mathscr{X})$, $\mathbb{X}$ can be canonically defined as
\begin{equation*}
	\mathbb{X}_{s,t}=\int_{s}^{t}X_{s,r}dX_r,
\end{equation*}
where the integral is understood as Riemann–Stieltjes integral and it is easily checked that $(X,\mathbb{X}) \in\mathscr{C}^{p-var}([0,1], \mathscr{X})$ for any $p\geq 1$. For convenience, we often write the rough lift in this particular case as
\begin{align*}
	S_2: C^{1-var}([0,1],\mathscr{X}) &\rightarrow \mathscr{C}^{p-var}([0,1], \mathscr{X})\\
	X&\mapsto \boldsymbol{X}.
\end{align*}

  A rough path $\boldsymbol{X}\in \mathscr{C}^{p-var}([0,1], \mathscr{X})$ is said to be a geometric  rough path if we can find a sequence $\{X^k\}_{k\geq 1}\subset C^{1-var}([0,1],\mathscr{X})$ such that
\begin{equation}
	\label{Geometric rough path}
	\lim_{k\rightarrow \infty}\norm{\boldsymbol{X}-S_2(X^k)}_{p-var;[0,1]}\rightarrow 0.
\end{equation}
If $X^k$ can be chosen to be the piece-wise linear approximations of $X$ along a sequence of increasing partitions, then we often call $\boldsymbol{X}$ the natural rough lift of $X$. Another important property of geometric rough paths is the ``first order calculus''
\begin{equation}
	\label{First order calculus}
	Symm(\mathbb{X}_{s,t})=\frac{1}{2}X_{s,t}\otimes X_{s,t},
\end{equation}
where $Symm$ denotes the symmetric part. We denote the space of geometric rough paths by $G\mathscr{C}^{p-var}([0,1], \mathscr{X})$.

For a given geometric rough path $\boldsymbol{X}$, RDE of the form
\begin{equation*}
	dY_t=V(Y_t)d\boldsymbol{X}_t+V_0(Y_t)dt,\; Y_0=y_0\in\mathbb{R}^d, \; t\in [0,1],
\end{equation*}
where $V,V_0 \in C^2_b(V,L(V, V))$, always admits a unique solution. It is customary to write
\[ Y_t=\pi_V(0,y_0; \boldsymbol{X})(t), \ t\in[0,1],  \]
where $\pi$ is called the It\^o-Lyons map. Moreover, we have
\begin{equation*}
	\lim_{k\rightarrow \infty}\norm{\pi_V(0,y_0; \boldsymbol{X})-\pi_V(0,y_0; \boldsymbol{X}^k)}_{p-var;[0,1]}=0,
\end{equation*} 
where $\{\boldsymbol{X}^k\}_{k\geq 1}=\{S_2(X^k)\}_{k\geq 1}$ is any sequence with finite 1-variation that satisfy \eqref{Geometric rough path}.

There are two important quantities we will use later in our study. For $\alpha>0$, we define the accumulated $\alpha$-local variation as 
\[M_{\alpha}( \boldsymbol{X})=\sup_{\substack{(t_i,t_{i+1})\in D([0,1]) \\ \norm{\boldsymbol{X}}^p_{p-var;[t_i,t_{i+1}]}\leq \alpha}}\sum_{i}\norm{\boldsymbol{X}}^p_{p-var;[t_i,t_{i+1}]}.  \]
Next, we have the greedy sequence:
\begin{align*}
	\tau_0&=0\\
	\tau_{i+1}&=\inf\{t\in (\tau_{i}, 1]: \norm{\boldsymbol{X}}^p_{p-var;[\tau_{i},\tau_{i+1}]}=\alpha   \}\wedge 1,
\end{align*}
with the convention that $\inf\emptyset=+\infty$. Associated with it is
\[N_\alpha(\boldsymbol{X}):=\sup\{a\in\mathbb{N}\cup \{0\}: \tau_{a}< 1 \}.  \]
The following bound is always true
\begin{equation}
	\label{Boung on N}
	\alpha N_\alpha(\boldsymbol{X})\leq \norm{\boldsymbol{X}}^p_{p-var;[0,1]}.
\end{equation}
Moreover, we have the following relation (see proposition 4.11 of \cite{MR3112937})
\[ M_{\alpha}( \boldsymbol{X})\leq \alpha(2N_\alpha(\boldsymbol{X})+1).  \]

\section{Rough lift of $X_t$ }
\subsection{Piece-wise linear approximation}
In this section we study the rough lift of process $X_t\in\mathbb{R}^d$, whose components are independent copies of $I_n(f_t)$. Let $\{Q^l\}_{l\geq 1}$ be a sequence of increasing partition of $[0,1]$, whose mesh goes to zero as $n$ tends to infinity. We prepare a general lemma before our main results.
\begin{lemma}[Proposition 5.60 of \cite{friz2010multidimensional}]
	\label{Piece-wise p-var}
	Let $p>1$ and $F\in C([0,1]^2)$ with finite 2D $p$-variation. Then for any partitions $Q,\tilde{Q}$ of $[0,1]$ and $0\leq s,t,u,v\leq 1$, we have
	 \[ \norm{F^{Q,\tilde{Q}}}_{p-var;[s,t]\times [u,v]}\preceq \norm{F}_{p-var;[s,t]\times [u,v]},   \]
	 where $F^{Q,\tilde{Q}}$ is the piece-wise linear type approximation of $F$ along $Q,\tilde{Q}$.
\end{lemma}

Our next result gives a natural rough lift of $X_t$.
\begin{theorem}
	\label{Piece-wise linear rough lift}
	Suppose that assumption \ref{Assumption 1} holds for the components of $X_t$. Let $X^l_t$ be the piece-wise linear approximation of $X_t$ along $\{Q^l\}_{l\geq 1}$ and $\boldsymbol{X}^l$ be the natural lift of $X^l_t$, then for any $\rho'\in (\rho, 3/2)$ and $p>1$, we have 
	\begin{align*}
		\lim_{l,m\rightarrow \infty}\mathbb{E}\left(d_{2\rho'-var; [0,1]}(\boldsymbol{X}^l, \boldsymbol{X}^m )   \right)^p=0.
	\end{align*}
	In particular, $X_t$ admits a natural geometric rough lift $\boldsymbol{X}_t$ and 
	\begin{equation*}
		\mathbb{E}\left( \norm{\boldsymbol{X}_t}_{2\rho'-var;[0,1]} \right)^p<+\infty.
	\end{equation*}
\end{theorem}

	\begin{proof}
		By theorem A.13 of \cite{friz2010multidimensional}, it is enough to show that for any $s,t\in[0,1]$, $1\leq i,j \leq d$ and $p>1$
		\begin{align}
			\sup_{l\geq 1}\mathbb{E}\abs{X^{l;i}_{s,t}}^p&\preceq  \abs{t-s}^{\frac{p}{2\rho'}} \label{level 1 bound}\\
			\sup_{l\geq 1}\mathbb{E}\abs{\int_{s}^{t}X^{l;i}_{s,r}dX^{l;j}_r}^p&\preceq   \abs{t-s}^{\frac{p}{\rho'}},\label{level 2 bound}\\
			\mathbb{E}\abs{X^{l;i}_{s,t}-X^{m;i}_{s,t}}^p&\preceq \epsilon \abs{t-s}^{\frac{p}{2\rho'}}\label{level 1} \\
			\mathbb{E}\abs{\int_{s}^{t}X^{l;i}_{s,r}dX^{l;j}_r-\int_{s}^{t}X^{m;i}_{s,r}dX^{m;j}_r}^p&\preceq  \epsilon \abs{t-s}^{\frac{p}{\rho'}},\label{level 2}
		\end{align}
		where $X^{l;i}$ is the $i$-th component of $X^l$, similarly for all other terms, and $\epsilon$ is some positive number that goes to zero as $l,m$ tend to infinity. We will only prove \eqref{level 1} and \eqref{level 2}, since \eqref{level 1 bound} and \eqref{level 2 bound} can be derived with exact same (in fact easier) arguments.
		
		 For level 1 and any $1\leq i\leq d$, we have by moment equivalence
		\begin{align*}
			\mathbb{E}\abs{X^{l;i}_{s,t}-X^{m;i}_{s,t}}^p&\preceq \left(\mathbb{E}\abs{X^{l;i}_{s,t}-X^{m;i}_{s,t}}^2 \right)^{\frac{p}{2}}\preceq \norm{R_{X^l-X^m}}_{\rho'-var; [s,t]^2}^{\frac{p}{2}},
		\end{align*}
	where $R_{X^l-X^m}$ is the covariance function of $X^l-X^m$, which can be written in terms of piece-wise linear approximation of $R_X$ along $Q^l,Q^m$. Since $\rho'>\rho$, we have by interpolation and lemma \ref{Piece-wise p-var} that
	\begin{align}
		\norm{R_{X^l-X^m}}_{\rho'-var; [s,t]^2}^{\frac{p}{2}}&\preceq \norm{R_{X^l-X^m}}^{\frac{\rho'-\rho}{\rho'}\cdot \frac{p}{2}}_{\infty;[0,1]^2}\norm{R_{X^l-X^m}}^{\frac{\rho}{\rho'}\cdot \frac{p}{2}}_{\rho-var; [s,t]^2}\nonumber\\
		& \preceq  \norm{R_{X^l-X^m}}^{\frac{\rho'-\rho}{\rho'}\cdot \frac{p}{2}}_{\infty;[0,1]^2}\norm{R_{X}}^{\frac{\rho}{\rho'}\cdot \frac{p}{2}}_{\rho-var; [s,t]^2}\nonumber \\
		&\preceq  \norm{R_{X^l-X^m}}^{\frac{\rho'-\rho}{\rho'}\cdot \frac{p}{2}}_{\infty;[0,1]^2}\abs{t-s}^{\frac{p}{2\rho'}}.\label{level 1 final}
	\end{align}
	We used assumption \ref{Assumption 1} in the last inequality.
	
		For level 2, we may first consider
		\begin{equation*}
			\mathbb{E}\left(\int_{s}^{t}X^{l;i}_{s,r}dX^{l;j}_r-\int_{s}^{t}X^{m;i}_{s,r}dX^{m;j}_r\right)^2.
		\end{equation*}
		When $i=j$, by simple calculus, we have 
		\[\int_{s}^{t}X^{l;i}_{s,r}dX^{l;i}_r=\frac{1}{2}(X^{l;i}_t-X^{l;i}_s)^2=\frac{1}{2}(X^{l;i}_{s,t})^2.   \]
		This together with moment equivalence give
		\begin{align}
			\mathbb{E}\left(\int_{s}^{t}X^{l;i}_{s,r}dX^{l;j}_r-\int_{s}^{t}X^{m;i}_{s,r}dX^{m;j}_r\right)^2&=\frac{1}{4}\mathbb{E}\left( (X^{l;i}_{s,t}-X^{m;i}_{s,t} )(X^{l;i}_{s,t}+X^{m;i}_{s,t} )  \right)^2\nonumber\\
			&\preceq  \mathbb{E}(X^{l;i}_{s,t}-X^{m;i}_{s,t} )^2\mathbb{E}(X^{l;i}_{s,t}+X^{m;i}_{s,t} )^2 \nonumber\\
			&\preceq \norm{R_{X^l-X^m}}_{\rho'-var; [s,t]^2} \norm{R_X}_{\rho-var;[s,t]\times [s,t]}\nonumber\\
			&\preceq \norm{R_{X^l-X^m}}^{\frac{\rho'-\rho}{\rho'}}_{\infty;[0,1]^2} \abs{t-s}^{\frac{2}{\rho'}}. \label{level 2 piece 1}
		\end{align}
		When $i\neq j$, $X^i$ and $X^j$ are independent. Hence, we can write
		\begin{align}
			&\mathbb{E}\left(\int_{s}^{t}X^{l;i}_{s,r}dX^{l;j}_r-\int_{s}^{t}X^{m;i}_{s,r}dX^{m;j}_r\right)^2\nonumber\\
			=&\int_{[s,t]\times [s,t]}\mathbb{E}(X^{l;i}_{s,r}-X^{m;i}_{s,r})(X^{l;i}_{s,u}-X^{m;i}_{s,u})d\mathbb{E}(X^{l;j}_rX^{l;j}_u)\nonumber\\
			+&\int_{[s,t]\times [s,t]}\mathbb{E}(X^{m;i}_{s,r}X^{m;i}_{s,u})d\mathbb{E}(X^{l;i}_{s,r}-X^{m;i}_{s,r})(X^{l;i}_{s,u}-X^{m;i}_{s,u})\nonumber\\
			\preceq&  \norm{R_{X^l-X^m}}_{\rho'-var; [s,t]^2} \norm{R}_{\rho-var;[s,t]\times [s,t]}\preceq \norm{R_{X^l-X^m}}^{\frac{\rho'-\rho}{\rho'}}_{\infty;[0,1]^2} \abs{t-s}^{\frac{2}{\rho'}}.\label{level 2 piece 2}
		\end{align}
		Combining \eqref{level 2 piece 1} and \eqref{level 2 piece 2} and another application of moment equivalence gives
		\begin{equation}
			\mathbb{E}\abs{\int_{s}^{t}X^{l;i}_{s,r}dX^{l;j}_r-\int_{s}^{t}X^{m;i}_{s,r}dX^{m;j}_r}^p\preceq  \norm{R_{X^l-X^m}}^{\frac{\rho'-\rho}{\rho'}\cdot \frac{p}{2}}_{\infty;[0,1]^2} \abs{t-s}^{\frac{p}{\rho'}}.\label{level 2 final}
		\end{equation}		
		Putting \eqref{level 1 final}, \eqref{level 2 final} together, we have proved \eqref{level 1} and \eqref{level 2} with 
		\[ \epsilon=\norm{R_{X^l-X^m}}^{\frac{\rho'-\rho}{\rho'}\cdot \frac{p}{2}}_{\infty;[0,1]^2}. \]
		Since
		\[\lim_{l,m\rightarrow \infty}\norm{R_{X^l-X^m}}_{\infty;[0,1]^2}= 0, \]
		we conclude
		\begin{align*}
			\lim_{l,m\rightarrow \infty}\mathbb{E}\left(d_{2\rho'-var; [0,1]}(\boldsymbol{X}^l, \boldsymbol{X}^m )   \right)^p=0.
		\end{align*}
\end{proof}

\subsection{Karhunen-Lo\`eve type approximation}

Similar to Gaussian processes, we may introduce the following Karhunen-Lo\`eve type approximation to $X_t$. 

Let $\{\varphi_i \}_{i\geq 1}$ be an orthonormal basis of $\mathcal{H}^{\hat{\otimes} n}$, then it is readily checked that 
\begin{equation}
	\label{Karhunen-Loeve}
	I_n(f_t)=\sum_{i\geq 1}\langle f_t, \varphi_i \rangle I_n(\varphi_i),
\end{equation}
where the above series is convergent in $L^2(\Omega)$. We show that \eqref{Karhunen-Loeve} also gives a rough lift of $X_t$, which coincide with the one we constructed in theorem \ref{Piece-wise linear rough lift} almost surely. 

Our strategy is apparent, we show the partial sum $X_t^k$ whose components are independent copies of $\sum_{i= 1}^{k}\langle f_t, \varphi_i \rangle I_n(\varphi_i)$ has a natural rough lift $\boldsymbol{X}^k_t$ which converges to $\boldsymbol{X}_t$ as $k$ tends to infinity. We start with a lemma, reminiscent to the Cameron-Martin embedding of Gaussian measure.
\begin{lemma}
	\label{Embedding}
	Suppose assumption \ref{Assumption 1} holds. For any $\varphi \in \mathcal{H}^{\hat{\otimes} n}$, the process $\phi_t=\langle f_t, \varphi\rangle_{\mathcal{H}^{\otimes n}}$ has finite $\rho$-variation and
	\[ \norm{\phi_t}_{\rho-var;[0,1]}\preceq \norm{\varphi}_{\mathcal{H}} \sqrt{\norm{R}_{\rho-var;[0,1]}}.   \]	
\end{lemma}
\begin{proof}
	Recall we have the following isometry
	\[ R([s,t]\times [u,v])=\mathbb{E}(X_{s,t}X_{u,v})=n!\langle f_{s,t}, f_{u,v} \rangle_{\mathcal{H}^{\otimes n}}.  \]
	Let $\rho'$ be the H\"older conjugate of $\rho$. Then, for any partition of $[0,1]$, we have by duality that
	\begin{align*}
		&\left(\sum_{i}\abs{\langle f_{t_i, t_{i+1}}, \varphi\rangle}^{\rho} \right)^{\frac{1}{\rho}}=\sup_{\norm{\beta}_{l^{\rho'}}\leq 1}\sum_{i}\beta_i \langle f_{t_i, t_{i+1}}, \varphi\rangle_{\mathcal{H}^{\otimes n}}
		=\sup_{\norm{\beta}_{l^{\rho'}}\leq 1}\langle \sum_{i}\beta_if_{t_i, t_{i+1}}, \varphi\rangle_{\mathcal{H}^{\otimes n}}\\
		&\leq \sup_{\norm{\beta}_{l^{\rho'}}\leq 1}\norm{\varphi}_{\mathcal{H}}\sqrt{\sum_{i}\sum_{j}\beta_i\beta_j\langle f_{t_i,t_{i+1}}, f_{t_j,t_{j+1}}\rangle_{\mathcal{H}^{\otimes n}} }\\
		&\leq \sup_{\norm{\beta}_{l^{\rho'}}\leq 1} \norm{\varphi}_{\mathcal{H}} \left( \sum_{i}\sum_{j}\abs{\beta_i}^{\rho'}\abs{\beta_j}^{\rho'} \right)^{\frac{1}{2\rho'}} \left(  \sum_{i}\sum_{j}\abs{\langle f_{t_i,t_{i+1}}, f_{t_j,t_{j+1}}\rangle_{\mathcal{H}^{\otimes n}}}^\rho  \right)^{\frac{1}{2\rho}}\\
		&\preceq\norm{\varphi}_{\mathcal{H}} \left(  \sum_{i}\sum_{j}\abs{R([t_i,t_{i+1}]\times [t_j,t_{j+1}]) }^\rho  \right)^{\frac{1}{2\rho}}\leq \norm{\varphi}_{\mathcal{H}} \sqrt{\norm{R}_{\rho-var;[0,1]^2}}.
	\end{align*}
	Taking the supreme over all partitions concludes the proof.	
\end{proof}

\begin{lemma}
	\label{K-L approximation}
	Suppose assumption \ref{Assumption 1} holds for the components of $X_t$. Define $X_t^k$ to be the process whose components are independent copies of
	\[\sum_{i= 1}^{k}\langle f_t, \varphi_i \rangle I_n(\varphi_i).  \]
	Then, for each $k\geq 1$, $X^k_t$ admits a geometric rough lift $\boldsymbol{X}^k$.
\end{lemma}
\begin{proof}
	For a fixed $k\geq 1$, we have
	\begin{align*}
		R_{X^k}(s,t)=\mathbb{E}( X^k_tX^k_s)=\sum_{i= 1}^{k}\langle f_s, \varphi_i \rangle\langle f_t, \varphi_i \rangle.
	\end{align*}
	By triangle inequality, one sees that
	\begin{align*}
		\norm{R_{X^k}}_{\rho-var;[0,1]^2}&\leq \sum_{i= 1}^{k}\norm{\langle f_s, \varphi_i \rangle\langle f_t, \varphi_i \rangle}_{\rho-var;[0,1]^2}\leq \sum_{i= 1}^{k}\norm{\langle f_t, \varphi_i \rangle}^2_{\rho-var;[0,1]}\\
		&\leq  k\norm{R}_{\rho-var;[0,1]^2},
	\end{align*}
	where we used the previous lemma in the last step. As a result, for each fixed $k\geq 1$, $X^k_t$ is a process living in the $n$-th homogeneous chaos whose covariance function $R_{X^k}$ satisfies assumption \ref{Assumption 1}. Now our result follows from theorem \ref{Piece-wise linear rough lift}. 
\end{proof}
	To achieve our goal of proving that $\boldsymbol{X}^k$ converges to $\boldsymbol{X}$, the main ingredient is to get, in a certain sense, a uniform boundedness of $\boldsymbol{X}^k$. The crude estimate we got from previous lemma using triangle inequality is obviously not uniform. A more delicate projection argument is needed. 
	
\begin{proposition}
	Suppose assumption \ref{Assumption 1} holds for the components of $X_t$. Let $\boldsymbol{X}^k$ be the natural lifts we constructed in the previous lemma and $\boldsymbol{X}$ the rough lift we constructed in theorem \ref{Piece-wise linear rough lift}.  Then, for any $\rho'>\rho$ and $p>1$, we have 
	\begin{align*}
		\lim_{k\rightarrow \infty}\mathbb{E}\left(d_{2\rho'-var; [0,1]}(\boldsymbol{X}^k, \boldsymbol{X} )   \right)^p=0.
	\end{align*}  
\end{proposition}	
\begin{proof}
	It is sufficient to show that for any $s,t\in [0,1]$ and $1\leq i,j\leq d$
	\begin{align}
		\label{L^2 convergence k}
		\lim_{k\rightarrow \infty}X^{k;i}_{t}= X^i_t,\ \lim_{k\rightarrow \infty}\mathbb{X}^{k;i,j}_{t}= \mathbb{X}^{i,j}_{t},
	\end{align}
	in $L^2(\Omega)$ and 
	\begin{align}
		&\sup_{k}\mathbb{E}(X^{k,i}_{s,t})^2\leq \mathbb{E}(X^i_{s,t})^2\leq C\abs{t-s}^{\frac{1}{\rho'}} \label{level 1 Holder}\\
		&\sup_{k}\mathbb{E}(	\mathbb{X}^{k;i,j}_{s,t})^2\leq \mathbb{E}(	\mathbb{X}^{i,j}_{s,t})^2\leq C\abs{t-s}^{\frac{2}{\rho'}}.\label{level 2 Holder}
	\end{align}
	Indeed, by using moment equivalence and Garsia–Rodemich–Rumsey estimate (see, for instance theorem A.12 of \cite{friz2010multidimensional}) equations \eqref{level 1 Holder} and \eqref{level 2 Holder} imply that for any $p>1$
	\begin{equation*}
		\sup_{k}\mathbb{E}\norm{\boldsymbol{X}^k}^p_{1/(2\rho'')-\text{H\"ol}}<+\infty,
		\end{equation*}	
	where $\rho''\in (\rho', 3/2)$. By passing to a sub-sequence, equation \eqref{L^2 convergence k} implies $\boldsymbol{X}^k$ converges to $\boldsymbol{X}$ in probability. Hence, by theorem A.15 of \cite{friz2010multidimensional}, we conclude
	\[ \lim_{k\rightarrow \infty}d_{2\rho''-var;[0,1]}(\boldsymbol{X}^k, \boldsymbol{X})\leq \lim_{k\rightarrow \infty}d_{\frac{1}{2\rho''}-\text{H\"ol}}(\boldsymbol{X}^k, \boldsymbol{X})= 0 \]
	in $L^p(\Omega)$ for all $p>1$. 
	The rest of our proof will focus on establishing \eqref{L^2 convergence k}\eqref{level 1 Holder}\eqref{level 2 Holder}. We divide it into several steps.\newline
	\textbf{Step 1}:
	Let 
	\[  L^2_k=\overline{Span}\{ I_n(\varphi_1), \cdots, I_n(\varphi_k)   \},  \]
	which is a closed subspace of the $n$-th homogeneous Wiener space. We denote $P_{L^2_k}$ the projection onto $L^2_k$. It is then obvious that
	\begin{align*}
		X^k_t=P_{L^2_k}(X_t).
	\end{align*}
	By properties of projection and the completeness of orthonormal basis, one has
	\[ \lim_{k\rightarrow \infty}X^{k;i}_{t}= X^i_t,\ \text{in}\ L^2(\Omega),   \]
	and
	\[ \sup_{k}\mathbb{E}(X^{k,i}_{s,t})^2\leq \mathbb{E}(X^i_{s,t})^2\leq C\abs{t-s}^{\frac{1}{\rho'}}. \]
	This finishes the level 1 estimates of \eqref{L^2 convergence k}\eqref{level 1 Holder}\eqref{level 2 Holder}.\newline
	\textbf{Step 2}:
	Let $\mathbb{X}, \mathbb{X}^{k}$ be the second level processes of $\boldsymbol{X}, \boldsymbol{X}^k$ respectively. We first consider the case where $1\leq i\neq j \leq d$. 
	
	Observe that, for two random variables $x,y\in L^2(\Omega)$ that are independent, we have
	$x\cdot y\in L^2(\Omega)$. On the other hand, for $x\otimes y$ as an element in $L^2(\Omega)\otimes L^2(\Omega)$, we have
	\[ \norm{x\otimes y}^2_{L^2(\Omega)\otimes L^2(\Omega)}=\mathbb{E}(x^2)\mathbb{E}(y^2)=\mathbb{E}(x\cdot y)^2=\norm{x\cdot y}^2_{L^2(\Omega)}.  \]
	This isometry allows us to identify $x\cdot y$ with $x\otimes y$.
	
	Back to our proof, we claim that
	\begin{equation}
		\label{Assertion}
		\mathbb{X}_{s,t}^{k;i,j}=P_{L^2_k\otimes L^2_k }(\mathbb{X}_{s,t}^{i,j}).
	\end{equation}
	To see this, let $X^{l;i}_t, X^{k,l;i}_t$ be the $i$-th component of the piece-wise linear approximation of $X_t, X^k_t$ along the partition $Q^l$ respectively. Then, we have
	\[ P_{L^2_k\otimes L^2_k }\left(\int_{s}^{t}X^{l;i}_{s,r}dX^{l;j}_r \right)=\int_{s}^{t}X^{k,l;i}_{s,r}dX^{k,l;j}_r. \]
	Notice that we have identified the integral as an element in $L^2(\Omega)\otimes L^2(\Omega)$ thanks to the independence of $X^i$ and $X^j$. The isometry we discussed before can be written as 
	\begin{align}
		\label{Isometry}
		\mathbb{E}\left(\int_{s}^{t}X^{l;i}_{s,r}dX^{l;j}_r \right)^2=\norm{\int_{s}^{t}X^{l;i}_{s,r}dX^{l;j}_r}^2_{L^2(\Omega)\otimes L^2(\Omega)}.
	\end{align}
	Combining this with theorem \ref{Piece-wise linear rough lift} gives
	\begin{align*}
		&\lim_{l\rightarrow \infty}\norm{\int_{s}^{t}X^{k,l;i}_{s,r}dX^{k,l;j}_r-\mathbb{X}_{s,t}^{k;i,j}}_{L^2(\Omega)\otimes L^2(\Omega)}\\
		=&\lim_{l\rightarrow \infty}\norm{\int_{s}^{t}X^{k,l;i}_{s,r}dX^{k,l;j}_r-\mathbb{X}_{s,t}^{k;i,j}}_{L^2(\Omega)}=0. 
	\end{align*}
	We thus deduce from the continuity of projection that
	\begin{align*}
		\mathbb{X}^{k;i,j}_{s,t}&=\lim_{\substack{l\rightarrow \infty \\ \text{in}\ L^2(\Omega)}}\int_{s}^{t}X^{k,l;i}_{s,r}dX^{k,l;j}_r\\
		&=\lim_{\substack{l\rightarrow \infty \\ \text{in}\ L^2(\Omega)\otimes L^2(\Omega)}}P_{L^2_k\otimes L^2_k }\left(\int_{s}^{t}X^{l;i}_{s,r}dX^{l;j}_r \right)=P_{L^2_k\otimes L^2_k }\left(\lim_{\substack{l\rightarrow \infty \\ \text{in}\ L^2(\Omega)\otimes L^2(\Omega)}}\int_{s}^{t}X^{l;i}_{s,r}dX^{l;j}_r \right)\\
		&=P_{L^2_k\otimes L^2_k }\left(\lim_{\substack{l\rightarrow \infty \\ \text{in}\ L^2(\Omega)}}\int_{s}^{t}X^{l;i}_{s,r}dX^{l;j}_r \right)\\
		&=P_{L^2_k\otimes L^2_k }(\mathbb{X}^{i,j}_{s,t}),
	\end{align*}
	and our assertion \eqref{Assertion} follows. 
	
	It then follows that
	\begin{align*}
		\lim_{k\rightarrow \infty}\mathbb{E}(\mathbb{X}^{k;i,j}_{s,t}-\mathbb{X}^{i,j}_{s,t})^2&=\lim_{k\rightarrow \infty}\norm{\mathbb{X}^{k;i,j}_{s,t}-\mathbb{X}^{i,j}_{s,t}}^2_{L^2(\Omega)\otimes L^2(\Omega)}\\
		&=\lim_{k\rightarrow \infty}\norm{P_{L^2_k\otimes L^2_k }(\mathbb{X}_{s,t}^{i,j})-\mathbb{X}^{i,j}_{s,t}}^2_{L^2(\Omega)\otimes L^2(\Omega)}=0,
	\end{align*}
	and
	\begin{align*}
		\sup_k\mathbb{E}(\mathbb{X}^{k;i,j}_{s,t})^2=\sup_k\norm{\mathbb{X}^{k;i,j}_{s,t}}^2_{L^2(\Omega)\otimes L^2(\Omega)}&=\sup_k\norm{P_{L^2_k\otimes L^2_k }(\mathbb{X}_{s,t}^{i,j})}^2_{L^2(\Omega)\otimes L^2(\Omega)}\\
		&\leq \norm{\mathbb{X}^{i,j}_{s,t}}^2_{L^2(\Omega)\otimes L^2(\Omega)}=\mathbb{E}(\mathbb{X}^{i,j}_{s,t})^2.
	\end{align*}
	This completes the level 2 estimates with $1\leq i\neq j\leq d$ in \eqref{L^2 convergence k}\eqref{level 1 Holder}\eqref{level 2 Holder}.\newline
	\textbf{Step 3}: For level 2 and $i=j$, we no longer have the independence between integrand and integrator. Hence the isometry \eqref{Isometry} fails. Instead, we can rely on explicit calculus computations. As we pointed out before, we have
	\[\mathbb{X}^{k;i,i}_{s,t}=\frac{1}{2}(X^{k;i}_{s,t})^2. \]
	As a result,
		\begin{align*}
		\lim_{k\rightarrow \infty}\mathbb{E}\left(\mathbb{X}^{k;i,i}_{s,t}-\mathbb{X}^{i,i}_{s,t}\right)^2&\preceq\lim_{k\rightarrow \infty}\sqrt{\mathbb{E}\abs{X^i_{s,t}}^4\mathbb{E}\abs{X^{k;i}_{s,t}-X^i_{s,t}}^4}\\
		&\preceq \lim_{k\rightarrow \infty}\mathbb{E}\abs{X^i_{s,t}}^2\mathbb{E}\abs{X^{k;i}_{s,t}-X^i_{s,t}}^2=0.
	\end{align*}
	Moreover,
	\begin{align*}
		\sup_k\mathbb{E}\left(\mathbb{X}^{k;i,i}_{s,t}\right)^2&\preceq \sup_k\mathbb{E}\left(X^{k;i}_{s,t}\right)^4\preceq \sup_k\left(\mathbb{E}(X^{k;i}_{s,t})^2\right)^2\\
		&\leq \left(\mathbb{E}(X^{i}_{s,t})^2\right)^2\leq C\abs{t-s}^{\frac{2}{\rho'}}.
	\end{align*}
	The proof is now complete.
\end{proof}

A rather straightforward consequence is the following
\begin{corollary}
	Suppose assumption \ref{Assumption 1} holds for $X_t=I_n(f_t)$. Define
	\[\mathcal{H}_X=\{S_2 (\langle f_t, h \rangle_{\mathcal{H}^{\hat{\otimes} n}}),\ h\in\mathcal{H}^{\hat{\otimes} n}   \},\]
	Then
	\[ Supp(X_t)\subset Cl(\mathcal{H}_X),  \]
	where the closure is taken in the $\alpha$-H\"older topology for $\alpha<1/(2\rho)$.
\end{corollary}
\begin{proof}
	This is immediate from the fact that each $X^k$ we defined in lemma \ref{K-L approximation} can be written as a linear combination of processes in the form $\langle f_t, h \rangle_{\mathcal{H}^{\hat{\otimes} n}}$ and the convergence we set up in the last proposition.

\end{proof}
\begin{remark}
	Unlike the case when $n=1$, in general, we do not have the reverse inclusion. To see this, let us consider the following simple example where  
	\[ f_t=t\varphi_1\hat{\otimes} \varphi_1. \] 
	Note that 
	\[  \{ \langle f_t, h \rangle,\ \forall h\in \mathcal{H}^{\hat{\otimes} 2} \}=\{ c\cdot \langle f_t, \varphi_1\hat{\otimes} \varphi_1 \rangle,\ c\in\mathbb{R} \}=\{c\cdot t,\ c\in\mathbb{R} \}.   \]
	But on the other hand, 
	\begin{equation*}
		Supp(I_2(f_t))=\{ c\cdot t ,\ c\geq -1  \}.
	\end{equation*}
	This is because 
	\[I_2(\varphi_1\hat{\otimes} \varphi_1)=\frac{1}{2}(I_1(\varphi_1))^2-1\geq -1.\]
	More generally, the boundedness of all even order Hermite polynomials (which appear as the coefficients of Karhunen-Lo\`eve approximation) prevents $X^k$ from hitting all possible linear combinations of $\{\langle f_t, h \rangle_{\mathcal{H}^{\hat{\otimes} n}}\}$. In addition, these coefficients, although orthonormal in $L^2(\Omega)$, are no longer independent due to the loss of Gaussianity, which makes it almost impossible to extract any more information.
\end{remark}
\subsection{Rough lift of $X_t$ together with its Malliavin derivatives}
In this subsection, let us consider 
\[ \hat{X}_t:=(X_t, DX_t, \cdots, D^nX_t)\in\mathbb{R}^d\times \mathcal{H}^d\times\cdots \times (\mathcal{H}^{\otimes n})^d,         \]
which now is a process living in an infinitely dimensional Banach space. This change from finite to infinite dimension has significant impact. To make our arguments as transparent as possible, let us first explain how we used finite dimension to our advantage in our previous constructions for the rough lift of $X_t$. 

In theorem \ref{Piece-wise linear rough lift}, we used the fact that for any $1\leq i\leq d$ 
\begin{align*}
	\mathbb{X}^{i,i}_{s,t}=&\left( \int_{s}^{t}X_{s,r}\otimes dX_r  \right)^{i,i}\\
	=&\int_{s}^{t}X^i_{s,r} dX^i_r\cdot b_i\otimes b_i=Symm\left(\int_{s}^{t}X^i_{s,r} dX^i_r\cdot b_i\otimes b_i \right),
\end{align*}
where $b_i$ is the $i$-th basis vector of $\mathbb{R}^d$ and $Symm$ denotes the symmetric part. In other words, $\mathbb{X}^{i,i}$ is symmetric. This is coming from the fact that each component $X^i_t$ is living in a one dimensional subspace of $\mathbb{R}^d$
\[ X^i_t=X^i_t\cdot b_i \in \overline{Span}\{b_i\},\ \forall t\in[0,1].  \]
Consequently, by the property \eqref{First order calculus} of geometric rough paths, we must define $\mathbb{X}^{i,i}_{s,t}=\frac{1}{2}(X^i_{s,t})^2$. There is no other choice.

However, when we consider $\hat{X}_t$, all the components involving Malliavin derivatives are not one dimensional anymore. They are living in Hilbert spaces of the form $\mathcal{H}^{\otimes k}$. In general
\begin{equation}
	\label{Infinite dimensional iterated integrals}
	\int_{s}^{t}D^kX^i_{s,r}\otimes dD^kX^i_r\neq Symm\left( \int_{s}^{t}D^kX^i_{s,r}\otimes dD^kX^i_r \right),\ k\geq 1.
\end{equation}
As a result, there is no canonical way of defining the iterated integrals appear in \eqref{Infinite dimensional iterated integrals}. 

To overcome this difficulty, we must directly work with the tensor products of $D^kX_t$ and show that integrals in \eqref{Infinite dimensional iterated integrals} converge in an appropriate sense. The price one needs to pay is to impose a condition that can control the regularity of these tensor products, which motivates our assumption \ref{Assumption 2}.

We prepare another technical lemma before our main result of this section.
\begin{lemma}
	\label{Inner product}
	For $f,g\in\mathcal{H}^{\hat{\otimes} n}$ and $1\leq k\leq n$, we have
	\[\langle D^kI_n(f), D^kI_n(g) \rangle_{\mathcal{H}^{\otimes k}}=\left( \frac{n!}{(n-k)!} \right)^2\sum_{r=0}^{n-k}r!\binom{n-k}{r}\binom{n-k}{r}I_{2n-2k-2r}(f\hat{\otimes}_{(r+k)}g).  \]
\end{lemma}
\begin{proof}
	Let $\{e_i\}_{i\geq 1}$ be an orthonormal basis of $\mathcal{H}$, then we have
	\begin{align*}
		&\langle D^kI_n(f), D^kI_n(g) \rangle_{\mathcal{H}^{\otimes k}}\\
		&=\sum_{i_1,\cdots, \i_k\geq 1}\langle D^kI_n(f), e_{i_1}\otimes\cdots \otimes e_{i_k} \rangle_{\mathcal{H}^{\otimes k}}\langle D^kI_n(g), e_{i_1}\otimes\cdots \otimes e_{i_k} \rangle_{\mathcal{H}^{\otimes k}}\\
		&=\left( \frac{n!}{(n-k)!} \right)^2\sum_{i_1,\cdots, \i_k\geq 1}\langle I_{n-k}(f), e_{i_1}\otimes\cdots \otimes e_{i_k} \rangle_{\mathcal{H}^{\otimes k}}\langle I_{n-k}(g), e_{i_1}\otimes\cdots \otimes e_{i_k} \rangle_{\mathcal{H}^{\otimes k}}\\
		&=\left( \frac{n!}{(n-k)!} \right)^2\sum_{i_1,\cdots, \i_k\geq 1}I_{n-k}(\langle f, e_{i_1}\otimes\cdots \otimes e_{i_k} \rangle_{\mathcal{H}^{\otimes k}})I_{n-k}(\langle g, e_{i_1}\otimes\cdots \otimes e_{i_k} \rangle_{\mathcal{H}^{\otimes k}}).
	\end{align*}
	To ease notations, we define
	\[ f^{i_1,\cdots, i_k}:=\langle f, e_{i_1}\otimes\cdots \otimes e_{i_k} \rangle_{\mathcal{H}^{\otimes k}}.  \]
	Then,
	\begin{equation*}
		\langle D^kI_n(f), D^kI_n(g) \rangle_{\mathcal{H}^{\otimes k}}=\left( \frac{n!}{k!} \right)^2\sum_{i_1,\cdots, \i_k\geq 1}I_{n-k}(f^{i_1,\cdots, i_k})I_{n-k}(g^{i_1,\cdots, i_k}).
	\end{equation*}
	By the product formula for Wiener chaos, we have
	\begin{align*}
		&I_{n-k}(f^{i_1,\cdots, i_k})I_{n-k}(g^{i_1,\cdots, i_k})\\
		&=\sum_{r=0}^{n-k}r!\binom{n-k}{r}\binom{n-k}{r}I_{2n-2k-2r}(f^{i_1,\cdots, i_k}\hat{\otimes}_rg^{i_1,\cdots, i_k}).
	\end{align*}
	By linearity of symmetrization, we conclude
	\begin{equation*}
		\langle D^kI_n(f), D^kI_n(g) \rangle_{\mathcal{H}^{\otimes k}}=\left( \frac{n!}{(n-k)!} \right)^2\sum_{r=0}^{n-k}r!\binom{n-k}{r}\binom{n-k}{r}I_{2n-2k-2r}(f\hat{\otimes}_{(r+k)}g).
	\end{equation*}
\end{proof}

We can prove our main result of this section now. As before, $\{Q^l\}_{l\geq 1}$ denotes a sequence of increasing partitions of $[0,1]$.
\begin{theorem}
	\label{Enhance piece-wise linear rough lift}
	Suppose assumption \ref{Assumption 2} holds for the components of $X_t$. Let $\hat{X}^l_t$ be the piece-wise linear approximation of $\hat{X}_t$ along $\{Q^l\}_{l\geq 1}$ and $\boldsymbol{\hat{X}}^l$ the natural lift of $\hat{X}^l_t$, then for any $\rho'\in (\rho, 3/2)$ and $p>1$, we have 
	\begin{align*}
		\lim_{l,m\rightarrow \infty}\mathbb{E}\left(d_{2\rho'-var; [0,1]}(\boldsymbol{\hat{X}}^l, \boldsymbol{\hat{X}}^m )   \right)^p=0.
	\end{align*}
	In particular, $\hat{X}_t$ admits a natural geometric rough lift $\boldsymbol{\hat{X}}_t$ and 
	\begin{equation*}
		\mathbb{E}\left( \norm{\boldsymbol{\hat{X}}_t}_{2\rho'-var;[0,1]} \right)^p<+\infty.
	\end{equation*}
\end{theorem}

\begin{proof}
		Let 
		\[ \hat{X}^{i,j}_t=D^iX^j_t,\ \hat{X}^{l;i,j}_t=D^iX^{l;j}_t,   \]
		where $X^{l;j}$ is the $j$-th component of $X^l_t$.
		We argue similarly to theorem \ref{Piece-wise linear rough lift} and it is enough to show that for any $s,t\in[0,1]$, $0\leq i, k\leq n,\ 1\leq j_1,j_2\leq d$ and $p>1$
	\begin{align}
		\mathbb{E}\norm{X^{l;i,j_1}_{s,t}-X^{m;i,j_1}_{s,t}}_{\mathcal{H}^{\otimes i}}^p&\preceq \epsilon \abs{t-s}^{\frac{p}{2\rho'}}, \label{Enhance level 1} \\
		\mathbb{E}\norm{\int_{s}^{t}X^{l;i,j_1}_{s,r}dX^{l;k,j_2}_r-\int_{s}^{t}X^{m;i,j_1}_{s,r}dX^{m;k,j_2}_r}_{\mathcal{H}^{\otimes(i+k) }}^p&\preceq  \epsilon \abs{t-s}^{\frac{p}{\rho'}},\label{Enhance level 2}
	\end{align}
	Notice that for any $0\leq k\leq n$, we have
	\[ \mathbb{E}\langle D^kX_t, D^kX_s \rangle_{\mathcal{H}^{\otimes k}}=\frac{n!}{(n-k)!}\mathbb{E}(X_sX_t)=\frac{n!}{(n-k)!} R_X(s,t). \]
	When $j_1\neq j_2$, we may take advantage of independence between different components and estimates \eqref{Enhance level 1} and \eqref{Enhance level 2} can be established in the exact same way as we did in theorem \ref{Piece-wise linear rough lift} with the same
	\[\epsilon=\norm{R_{X^l-X^m}}^{\frac{\rho'-\rho}{\rho'}\cdot \frac{p}{2}}_{\infty;[0,1]^2}.\]
	 We thus omit these details and focus only on \eqref{Enhance level 2} with $j_1=j_2$, which is where new ingredients of assumption \ref{Assumption 2} are used.
	
	 We first consider the case where $i=k=0,\ j_1=j_2$ (this is the finite dimensional case in our discussion at the beginning of this subsection). We have $\hat{X}^{l;0,j_1}=X^{l;j_1}$. Since in this case integrals are symmetric , we have by simple calculus that
	\begin{equation}
		\label{Its symmetric}
		\int_{s}^{t}\hat{X}^{l;0,j_1}_{s,r}d\hat{X}^{l;0,j_1}_r=\int_{s}^{t}X^{l;j_1}_{s,r}dX^{l;j_1}_r=\frac{1}{2}(X^{l;j_1}_t-X^{l;j_1}_s)^2=\frac{1}{2}(X^{l;j_1}_{s,t})^2.
	\end{equation}
	Just as theorem \ref{Piece-wise linear rough lift}, we have
	\begin{align}
		\mathbb{E}\left(\int_{s}^{t}X^{l;j_1}_{s,r}dX^{l;j_1}_r-\int_{s}^{t}X^{m;j_1}_{s,r}dX^{m;j_1}_r\right)^2&=\frac{1}{4}\mathbb{E}\left( (X^{l;j_1}_{s,t}-X^{m;j_1}_{s,t} )(X^{l;j_1}_{s,t}+X^{m;j_1}_{s,t} )  \right)^2\nonumber\\
		&\leq\frac{1}{4}  \mathbb{E}(X^{l;i}_{s,t}-X^{m;i}_{s,t} )^2\mathbb{E}(X^{l;i}_{s,t}+X^{m;i}_{s,t} )^2 \nonumber\\
		&\preceq \norm{R_{X^l-X^m}}_{\rho'-var; [s,t]^2} \norm{R}_{\rho-var;[s,t]\times [s,t]}\nonumber\\
		&\preceq \norm{R_{X^l-X^m}}^{\frac{\rho'-\rho}{\rho'}}_{\infty;[0,1]^2} \abs{t-s}^{\frac{2}{\rho'}}.\label{Enhance level 2 piece 1}
	\end{align}

	Next, for $i,k>0,\ j_1=j_2$, \eqref{Its symmetric} not longer works (this is the infinite dimensional case). Instead, we have
	\begin{align}
		\label{Discrete integral for Malliavin derivative}
		\norm{\int_{s}^{t}\hat{X}^{l;i,j_1}_{s,r}\otimes d\hat{X}^{l;k,j_1}_{r}}^2_{\mathcal{H}^{\otimes i}\otimes \mathcal{H}^{\otimes k}}=\int_{[s,t]\times [s,t]}\langle \hat{X}^{l;i,j_1}_{s,r}, \hat{X}^{l;i,j_1}_{s,u} \rangle_{\mathcal{H}^{\otimes i}}d\langle \hat{X}^{l;k,j_1}_{r}, \hat{X}^{l;k,j_1}_{u} \rangle_{\mathcal{H}^{\otimes k}}.
	\end{align}
	
	As a result,
	\begin{align}
		&\mathbb{E}\norm{\int_{s}^{t}\hat{X}^{l;i,j_1}_{s,r}\otimes d\hat{X}^{l;k,j_1}_{r}-\int_{s}^{t}\hat{X}^{m;i,j_1}_{s,r}\otimes d\hat{X}^{m;k,j_1}_{r}}^2_{\mathcal{H}^{\otimes i}\otimes \mathcal{H}^{\otimes k}}\nonumber\\
		&\preceq \mathbb{E}\norm{\int_{s}^{t}\hat{X}^{l;i,j_1}_{s,r}-\hat{X}^{m;i,j_1}_{s,r}\otimes d\hat{X}^{l;k,j_1}_{r}}^2_{\mathcal{H}^{\otimes i}\otimes \mathcal{H}^{\otimes k}}+\mathbb{E}\norm{\int_{s}^{t}\hat{X}^{m;i,j_1}_{s,r}\otimes d(\hat{X}^{m;k,j_1}_{r}-\hat{X}^{l;k,j_1}_{r})}^2_{\mathcal{H}^{\otimes i}\otimes \mathcal{H}^{\otimes k}}\nonumber\\
		&=\int_{s}^{t}\int_{s}^{t}\mathbb{E}\langle \hat{X}^{l;i,j_1}_{s,r_1}-\hat{X}^{m;i,j_1}_{s,r_1}, \hat{X}^{l;i,j_1}_{s,r_2}-\hat{X}^{m;i,j_1}_{s,r_2} \rangle_{\mathcal{H}^{\otimes i}}d\langle \hat{X}^{l;k,j_1}_{r_1}, \hat{X}^{l;k,j_1}_{r_2} \rangle_{\mathcal{H}^{\otimes k}} \label{Complicated sum} \\
		&+\int_{s}^{t}\int_{s}^{t}\mathbb{E}\langle \hat{X}^{m;i,j_1}_{s,r_1}, \hat{X}^{m;i,j_1}_{s,r_2} \rangle_{\mathcal{H}^{\otimes i}}d\langle \hat{X}^{m;k,j_1}_{r_1}-\hat{X}^{l;k,j_1}_{r_1}, \hat{X}^{m;k,j_1}_{r_2}-\hat{X}^{l;k,j_1}_{r_2} \rangle_{\mathcal{H}^{\otimes k}}\nonumber
	\end{align}
	By the orthogonality of Wiener chaos and lemma \ref{Inner product}, the right-hand side of \eqref{Complicated sum} can be written as
	\begin{align}
		&=\sum_{w=0}^{i \wedge k} C_{n,i,k,w}	\int_{s}^{t}\int_{s}^{t} \langle f^l_{s,r_1}-f^m_{s,r_1}\hat{\otimes}_{(n-w)} f^l_{s,r_2}-f^m_{s,r_2}, d(f^l_{r_1}\hat{\otimes}_{(n-w)} f^l_{r_2} ) \rangle_{\mathcal{H}^{\otimes 2w}}\nonumber\\
		&+\sum_{w=0}^{i \wedge k} C_{n,i,k,w}\int_{s}^{t}\int_{s}^{t} \langle f^m_{s,r_1}\hat{\otimes}_{(n-w)} f^m_{s,r_2},d (f^m_{r_1}-f^l_{r_1}\hat{\otimes}_{(n-w)} f^m_{r_2}-f^l_{r_2}) \rangle_{\mathcal{H}^{\otimes 2w}}\nonumber\\
		&\preceq \sum_{w=0}^{i \wedge k} \int_{s}^{t}\int_{s}^{t} \norm{f^l_{s,r_1}-f^m_{s,r_1}\hat{\otimes}_{(n-w)} f^l_{s,r_2}-f^m_{s,r_2}}_{\mathcal{H}^{\otimes 2r}}d\norm{f^l_{r_1}\hat{\otimes}_{(n-w)} f^l_{r_2}}_{\mathcal{H}^{\otimes 2w}}\label{The ugly sum}\\
		&+ \sum_{w=0}^{i \wedge k} \int_{s}^{t}\int_{s}^{t} \norm{f^m_{s,r_1}\hat{\otimes}_{(n-w)} f^m_{s,r_2}}_{\mathcal{H}^{\otimes 2r}}d\norm{f^m_{r_1}-f^l_{r_1}\hat{\otimes}_{(n-w)} f^m_{r_2}-f^l_{r_2}}_{\mathcal{H}^{\otimes 2w}}.\nonumber
	\end{align}
	Moreover, by assumption \ref{Assumption 2}, all functions
	\[ R^w_f(s,t)=f_s\hat{\otimes}_wf_t,\ 1\leq w\leq n  \]
	satisfy
	\[  \norm{R^w([s,t],[ u,v])}_{\mathcal{H}^{\otimes_{2(n-w)}}}\leq \omega^{\frac{1}{\rho}}([s,t]\times [u,v]).  \]
	Thus, we may control the first term on the right hand side of \eqref{The ugly sum} as a 2D Young's integral and get
	\begin{align*}
		&\int_{s}^{t}\int_{s}^{t} \norm{f^l_{s,r_1}-f^m_{s,r_1}\hat{\otimes}_{(n-w)} f^l_{s,r_2}-f^m_{s,r_2}}_{\mathcal{H}^{\otimes 2w}}d\norm{f^l_{r_1}\hat{\otimes}_{(n-w)} f^l_{r_2}}_{\mathcal{H}^{\otimes 2w}}\\
		&\leq \norm{R^{n-w}_{f^l-f^m}}_{\rho'-var;[s,t]^2}\norm{R^{n-w}_{f^l}}_{\rho-var;[s,t]^2}\\
		&\preceq\norm{R^{n-w}_{f^l-f^m}}^{\frac{\rho'-\rho}{\rho'}}_{\infty;[0,1]}\abs{t-s}^{\frac{2}{\rho'}},
	\end{align*}
	where we used interpolation and lemma \ref{Piece-wise p-var} in the last inequality. Same estimate works for the second term on the right-hand side of \eqref{The ugly sum} as well. We thus arrive at
	\begin{align}
		&\mathbb{E}\norm{\int_{s}^{t}\hat{X}^{l;i,j_1}_{s,r}\otimes d\hat{X}^{l;k,j_1}_{r}-\int_{s}^{t}\hat{X}^{m;i,j_1}_{s,r}\otimes d\hat{X}^{m;k,j_1}_{r}}^2_{\mathcal{H}^{\otimes i}\times \mathcal{H}^{\otimes k}}\nonumber\\
		&\preceq\norm{R^{n-w}_{f^l-f^m}}^{\frac{\rho'-\rho}{\rho'}}_{\infty;[0,1]}\abs{t-s}^{\frac{2}{\rho'}}.\label{Enhance level 2 piece 2}
	\end{align}
	Combining \eqref{Enhance level 2 piece 1} and \eqref{Enhance level 2 piece 2} and another application of moment equivalence of Wiener chaos gives
	\begin{align}
		&\mathbb{E}\norm{\int_{s}^{t}\hat{X}^{l;i,j_1}_{s,r}\otimes d\hat{X}^{l;k,j_1}_{r}-\int_{s}^{t}\hat{X}^{m;i,j_1}_{s,r}\otimes d\hat{X}^{m;k,j_1}_{r}}^p_{\mathcal{H}^{\otimes i}\times \mathcal{H}^{\otimes k}}\nonumber\\
		&\preceq \max_{1\leq w\leq n}\norm{R^{n-w}_{f^l-f^m}}^{\frac{\rho'-\rho}{\rho'}\cdot\frac{p}{2}}_{\infty;[0,1]}\abs{t-s}^{\frac{p}{\rho'}},\ \forall 0\leq i,k\leq n.\label{Enhance level 2 final}
	\end{align}
	Notice we used in the previous estimate that
	\[ \norm{R_{X^l-X^m}}_{\infty;[0,1]}= \norm{R^{n}_{f^l-f^m}}_{\infty;[0,1]}. \]
	Gathering \eqref{Enhance level 2 final} and the parts we omitted gives \eqref{Enhance level 1} and \eqref{Enhance level 2} with
	\[ \epsilon=\max_{1\leq w\leq n}\norm{R^{n-w}_{f^l-f^m}}^{\frac{\rho'-\rho}{\rho'}\cdot\frac{p}{2}}_{\infty;[0,1]}.  \]
	Our result now follows from
	\[ \lim_{l,m\rightarrow \infty}\max_{1\leq w\leq n}\norm{R^{n-w}_{f^l-f^m}}_{\infty;[0,1]}=0. \]
	The proof is complete.
\end{proof}
	We conclude this section with an example. It is well known that we have a wide variety of Gaussian processes that satisfy assumption \ref{Assumption 1}. We show that it is in fact, relatively easy to use these Gaussian processes to construct a chaos process $X_t=I_n(f_t),\ n>1,$ that satisfies assumption \ref{Assumption 2}. 
\begin{example}
	Let $Y^i_t=I_1(g^i_t),\ 1\leq i\leq n$ be $n$ independent Gaussian processes such that $\{g^i_t\}_{0\leq t\leq 1}$ has finite $\rho$-variation in $\mathcal{H}$ for $1\leq i\leq n$ and $\rho\in [1,3/2)$. Then, 
	\[X_t=\prod_{i=1}^{n}Y^i_t= I_n( g^1_t\hat{\otimes} g^2_t \hat{\otimes}\cdots \hat{\otimes} g^n_t )  \]
	satisfies assumption \ref{Assumption 2}. 
	\begin{proof}
		We only show the case for $n=2$, the general case is similar. Let $f_t=g^1_t\hat{\otimes} g^2_t$, a standard computation gives
		\begin{align*}
			f_{s,t}&=g^1_{s,t}\hat{\otimes} g^2_t+g^1_s\hat{\otimes}g^2_{s,t}\\
			f_{u,v}&=g^1_{u,v}\hat{\otimes} g^2_v+g^1_u\hat{\otimes}g^2_{u,v}.
		\end{align*}
		Then,
		\begin{align*}
			&\abs{f_{s,t}\hat{\otimes}_2 f_{u,v}}\\
			=&\abs{\langle g^1_{s,t}, g^1_{u,v} \rangle_{\mathcal{H}} \langle g^2_t ,g^2_v \rangle_{\mathcal{H}}+ \langle g^1_t ,g^1_v \rangle_{\mathcal{H}}\langle g^2_{s,t}, g^2_{u,v} \rangle_{\mathcal{H}}}\\
			\leq& \left(\sum_{i=1}^{2}\norm{g^i}_{\rho-var;[s,t]}\norm{g^i}_{\rho-var;[u,v]}\right)\sum_{i=1}^{2}\left( \norm{g^i}_{\rho-var;[0,1]}+\norm{g^i_0}_{\mathcal{H}}   \right)^2.
		\end{align*}
		Moreover,
		\begin{align*}
			&\norm{f_{s,t}\hat{\otimes}_1 f_{u,v}}_{\mathcal{H}^{\otimes 2}}\\
			=&\norm{ \langle g^1_{s,t}, g^1_{u,v} \rangle_{\mathcal{H}}\cdot g_t\hat{\otimes} g_v+ \langle g^2_{s,t}, g^2_{u,v} \rangle_{\mathcal{H}}\cdot g^1_u\hat{\otimes} g^1_s  }_{\mathcal{H}^{\otimes 2}}\\
			\leq& \left(\sum_{i=1}^{2}\norm{g^i}_{\rho-var;[s,t]}\norm{g^i}_{\rho-var;[u,v]}\right)\sum_{i=1}^{2}\left( \norm{g^i}_{\rho-var;[0,1]}+\norm{g^i_0}_{\mathcal{H}}   \right)^2.
		\end{align*}
		We conclude that $X_t$ satisfy assumption \ref{Assumption 2} with control
		\[ \omega([s,t]\times[u,v])=2^{\rho-1}\left(\sum_{i=1}^{2}\norm{g^i}^\rho_{\rho-var;[s,t]}\norm{g^i}^\rho_{\rho-var;[u,v]}\right)\sum_{i=1}^{2}\left( \norm{g^i}_{\rho-var;[0,1]}+\norm{g^i_0}_{\mathcal{H}}   \right)^{2\rho}.  \]
	\end{proof}
\end{example}

\section{Large deviation}
The large deviation principle for a single random variable in the $n$-th chaos is well known (see \cite{MR1600888}). Our goal in this section is to first establish a large deviation principle for a general chaos process $X_t$, then extend to its rough lift $\boldsymbol{X}_t$.

Let $\mathcal{S}=C([0,1], \mathbb{R}^d)$ equipped with the uniform topology. We define the following rate function on $\mathcal{S}$
\begin{equation*}
	I(x)=\begin{cases*}
		\inf\left\{\frac{1}{2}\norm{h}^2_{\mathcal{H}^d}\mid\ h\in\mathcal{H}^d,\ \langle f_t, (h^i)^{\otimes n} \rangle=x^i, 1\leq i\leq d \right\} ,\\
		+\infty\;\; \text{if no such $h$ exists.}
	\end{cases*}
\end{equation*}
 For a Borel set in $\mathcal{S}$, we set $I(A)=\inf_{x\in A}I(x)$. We first establish a large deviation principle for the process $X_t=I_n(f_t)\in\mathbb{R}$.
\begin{proposition}
	\label{LDP for level 1}
	Suppose assumption \ref{Assumption 1} holds for the components of $X_t$. The family of random processes $\{X_t(\epsilon\omega)\}_{\epsilon>0} $ satisfies a large deviation with
	 good rate function $I(x)$.
\end{proposition} 
\begin{proof}
	Without loss of generality, we assume $d=1$. Since $I_n(f_t)$ belongs to the $n$-th homogeneous chaos, by replacing $\epsilon$ with $1/r$, it is enough to show
	\begin{align*}
		\liminf_{r\rightarrow \infty}\frac{1}{r^2}\log\mathbb{P}(X_t\in r^nA )&\geq-I(A),\ \forall A\subset \mathcal{S}\ \text{open},\\
		\limsup_{r\rightarrow \infty}\frac{1}{r^2}\log\mathbb{P}(X_t\in r^nA )&\leq-I(A),\ \forall A\subset \mathcal{S}\ \text{closed}.
	\end{align*}
	We will split our proof into two parts. \newline
	\textbf{Lower Bound}: Let $A$ be an open subset of $\mathcal{S}$ and $B(\mathcal{H})$ be the unit ball in $\mathcal{H}$. Suppose we can find $h\in\mathcal{H}$ such that $\langle f_t, h^{\otimes n} \rangle\in A$ (if no such $h$ exists, the lower bound is trivially true).  Since $A$ is open, we can find $s>0$ such that 
	\[ V:=V\left(\langle f_t, h^{\otimes n} \rangle, s   \right)\in A,  \] 
	where $V(x,s)$ denotes the $s$-neighborhood of $x$. Thus,
	\begin{align*}
		\mathbb{P}(X_t\in r^nA )\geq \mathbb{P}(X_t\in r^nV )=\mathbb{P}\left(\sup_t\abs{X_t-r^n\langle f_t, h^{\otimes n} \rangle}<r^ns \right)=\mathbb{P}(U),
	\end{align*}
	where
	\[U=\left\{ \omega\in\Omega \mid \sup_t\abs{X_t(\omega)-\langle f_t, h^{\otimes n} \rangle}<r^ns   \right\}.  \]
	By Cameron-Martin theorem,
	\begin{align*}
		\mathbb{P}(U)=\int_{U-rh}\exp\{r\hat{h}-r^2\norm{h}^2_\mathcal{H}/2 \}\mathbb{P}(d\omega).
	\end{align*}
	An application of Jensen's inequality gives
	\begin{align}
		\label{Jensen}
		\mathbb{P}(U)\geq\exp\left\{ -r^2\norm{h}^2_\mathcal{H}/2 \right\}\mathbb{P}(U-rh)\exp\left\{\frac{r}{\mathbb{P}(U-rh)}\int_{U-rh}\hat{h} \mathbb{P}(d\omega) \right\}
	\end{align}
	By theorem \ref{Enhance piece-wise linear rough lift}, we know $\sup_t\norm{D^kX_t}\in L^p(\Omega)$ for all $p>1$. Thus,
	\begin{align*}
		\lim_{r\rightarrow \infty}	\frac{1}{r^n}\mathbb{E}\left(\sup_t\abs{X_t(x+rh)-r^n\langle f_t, h^{\otimes n} \rangle}  \right)
		=\lim_{r\rightarrow \infty}\frac{1}{r^n}\mathbb{E}\left( \sum_{k=0}^{n-1}r^k\sup_t\norm{D^kX_t}\norm{h}^k   \right)=0.
	\end{align*}
	We deduce that when $r$ is sufficiently large,
	\begin{align}
		\label{Lower bound 1}
		\mathbb{P}(U-rh)=\mathbb{P}\left( \sup_t\abs{X_t(x+rh)-r^n\langle f_t, h^{\otimes n} \rangle}<r^n s\right)>1/2.
	\end{align}
	On the other hand,
	\begin{align}
		\label{Lower bound 2}
		\frac{r}{\mathbb{P}(U-rh)}\int_{U-rh}\hat{h} \mathbb{P}(d\omega)\geq -\frac{r}{\mathbb{P}(U-rh)}\int_{U-rh}\mid \hat{h} \mid \mathbb{P}(d\omega)\geq -r \norm{h}_\mathcal{H}.
	\end{align}
	Combining \eqref{Jensen}, \eqref{Lower bound 1} and \eqref{Lower bound 2} gives that when $r$ is large enough,
	\begin{equation*}
		\mathbb{P}(U)\geq \frac{1}{2}\exp\left\{ -r^2\norm{h}^2_\mathcal{H}/2-r\norm{h}_{\mathcal{H}} \right\}.
	\end{equation*}
	Taking limit gives
	\begin{equation*}
		\liminf_{r\rightarrow \infty}\frac{1}{r^2}\log\mathbb{P}(X_t\in r^nA )\geq\liminf_{r\rightarrow \infty}\frac{1}{r^2}\log\mathbb{P}(U)\geq-\frac{1}{2}\norm{h}^2_{\mathcal{H}}.
	\end{equation*}
	Maximizing over $h$ gives
	\begin{equation*}
		\liminf_{r\rightarrow \infty}\frac{1}{r^2}\log\mathbb{P}(X_t\in r^nA )\geq-I(A),
	\end{equation*}
	which is the desired lower bound.\newline
	\textbf{Upper Bound}: Let $A$ be a closed set in $\mathcal{S}$ and take $0<\lambda<I(A)$. By the definition of $I(A)$, we have
	\[  \cup_{h\in B(\mathcal{H})}(2\lambda)^{\frac{n}{2}}\langle f_t, h^{\otimes n} \rangle \cap A=\emptyset.  \] 
	Since $B(\mathcal{H})$ is compact, the set 
	\[ K= \cup_{h\in B(\mathcal{H})}(2\lambda)^{\frac{n}{2}}\langle f_t, h^{\otimes n} \rangle \]
	is also compact. Hence, $K$ and $A$ are disjoint, and we can find $\eta>0$ such that
	\[ \{x\in \mathcal{S}\mid d(x,K)<(2\lambda)^{\frac{n}{2}}\eta\} \cap A=\emptyset.  \]
	Now, let us consider
	\[ V'=\{\omega\in\Omega \mid \sup_t\abs{X_t(\omega+rh)-r^n\langle f_t, h^{\otimes n} \rangle }<r^n\eta  \}.  \]
	Then we claim that
	\begin{align}
		\label{Isoperimetric set up}
		\mathbb{P}(X_t\in r^nA )\leq \mathbb{P}\{ \omega\in\Omega\setminus (V'+r\sqrt{2\lambda} B(\mathcal{H})) \}.
	\end{align}
	To see this, observe that when $\xi\in (V'+r\sqrt{2\lambda} B(\mathcal{H}))$, we can find $\tau\in V'$ and $h\in B(\mathcal{H})$ so that
	\begin{align*}
		\sup_t\abs{X_t(\xi)-r^n(2\lambda)^{\frac{n}{2}}\langle f_t, h^{\otimes n} \rangle} =\sup_t\abs{X_t(\tau+r\sqrt{2\lambda} h)-r^n(2\lambda)^{\frac{n}{2}}\langle f_t, h^{\otimes n} \rangle}< r^n(2\lambda)^{\frac{n}{2}}\eta,
	\end{align*}
	which implies
	\[ \xi\in \left\{ X_t\in r^n\{x\in \mathcal{S}\mid d(x,K)<(2\lambda)^{\frac{n}{2}}\eta\}\right\} \subset \Omega\setminus\left\{X_t\in r^nA \right\}.  \]
	From previous estimate, we know that when $r$ is large enough $P(V')>1/2$. By Gaussian isoperimetric inequality (see, for instance, theorem 4.33 of \cite{bogachev1998gaussian}), we have from \eqref{Isoperimetric set up} that
	\begin{align*}
		\mathbb{P}(X_t\in r^nA )\leq \frac{1}{2}\exp\{ -r^2\lambda\}.
	\end{align*}
	We can thus deduce
	\begin{equation*}
		\limsup_{r\rightarrow \infty}\frac{1}{r^2}\log\mathbb{P}(X_t\in r^nA )\leq-\lambda.
	\end{equation*}
	Finally, note that $\lambda<I(A)$ is arbitrary. Passing to the limit gives
	\begin{equation*}
		\limsup_{r\rightarrow \infty}\frac{1}{r^2}\log\mathbb{P}(X_t\in r^nA )\leq-I(A),
	\end{equation*}
	which is the desired upper bound.
\end{proof}

Having established a large deviation principle for $X_t$, we are going to use an extended contraction argument to prove a similar result for the rough lift $\boldsymbol{X}_t$ of $X_t$. For convenience, we shall write $\Phi_l$ for the piece-wise linear approximations along a sequence of increasing partitions $\{Q^l\}_{l\geq 1}$. 
\begin{lemma}
	\label{Exponential goodness}
	Suppose assumption \ref{Assumption 1} holds for the components of $X_t$. Let $\delta>0$ fixed. Then for any $p>2\rho$, we have
	\begin{equation*}
		\lim_{l\rightarrow \infty}\overline{\lim_{\epsilon\rightarrow 0}}\epsilon^2\log \mathbb{P}\bigg( d_{p-var}\left(S_2\circ \Phi_l (X_t(\epsilon x)), \boldsymbol{X}(\epsilon x)\right)>\delta \bigg)=-\infty.
	\end{equation*}
\end{lemma}
\begin{proof}
	By homogeneity, we have
	\[d_{p-var}\left(S_2\circ \Phi_l (X_t(\epsilon x)), \boldsymbol{X}(\epsilon x)\right)=\epsilon^n d_{p-var}\left(S_2\circ \Phi_l (X_t( x)), \boldsymbol{X}(x)\right). \]
	For any $q>1$, the moment equivalence of Wiener chaos gives
	\[ \mathbb{E}\abs{X_t}^q\leq (q-1)^{\frac{n}{2}}\mathbb{E}\abs{X_t}^2. \]
	Thus,	
	\begin{align*}
		\mathbb{P}\left\{d_{p-var}\left(S_2\circ \Phi_l (X_t( x)), \boldsymbol{X}(x)\right)>\frac{\delta}{\epsilon^n}  \right\}&\leq  \frac{ \epsilon^{nq}}{\delta^q}\mathbb{E}\bigg( d_{p-var}\left(S_2\circ \Phi_l (X_t( x)), \boldsymbol{X}(x)\right)  \bigg)^q\\
		\leq\frac{ \epsilon^{nq}}{\delta^q}&(q-1)^{\frac{nq}{2}}\mathbb{E}\bigg( d_{p-var}\left(S_2\circ \Phi_l (X_t( x)), \boldsymbol{X}(x)\right)  \bigg)^2.
	\end{align*}
	To ease notations, we will write
	\[ d_p(l):=\mathbb{E}\bigg( d_{p-var}\left(S_2\circ \Phi_l (X_t( x)), \boldsymbol{X}(x)\right)  \bigg)^2.  \]
Taking logarithm leads to
\begin{align*}
	\log 	\mathbb{P}&\left\{d_{p-var}\left(S_2\circ \Phi_l (X_t( x)), \boldsymbol{X}(x)\right)>\frac{\delta}{\epsilon^n}  \right\}\\
	&\leq q\log\left(\frac{\epsilon^n(q-1)^{\frac{n}{2}}}{\delta} d_p(l) \right).
\end{align*}
Choose $q=1/\epsilon^2$, we see that when $\epsilon$ is small enough
\begin{align*}
	\epsilon^2\log 	\mathbb{P}&\left\{d_{p-var}\left(S_2\circ \Phi_l (X_t( x)), \boldsymbol{X}(x)\right)>\frac{\delta}{\epsilon^n}  \right\}\\
	&\leq \log\left(\epsilon^n(\frac{1}{\epsilon^2}-1)^{\frac{n}{2}}\frac{d_p(l)}{\delta} \right)\leq \log\left(2\frac{d_p(l)}{\delta} \right).
\end{align*}
By theorem \ref{Piece-wise linear rough lift}, $\lim_{l\rightarrow \infty}d_p(l)=0$. Our result then follows.
\end{proof}

\begin{lemma}
	\label{Continuity property on level sets}
	Suppose assumption \ref{Assumption 1} holds for the components of $X_t$. For all $M>0$, and $p>2\rho$, we have
	\begin{equation}
		\label{Uniform on bounded sets}
		\lim_{l\rightarrow \infty}\sup_{\{x: I(x)<M\}}d_{p-var}\left(S_2\circ \Phi_l (x), S_2(x)    \right)=0.
	\end{equation}
\end{lemma}
\begin{proof}
By lemma \ref{Embedding} each $x\in \mathcal{S}$ with $I(x)<+\infty$ has finite $\rho$-variation. We have 
\begin{align*}
	\sup_{\{x: I(x)<M\}}\norm{x}_{p-var;[0,1]}&\preceq \sup_{\{x: I(x)<M\}}\norm{x}_{\rho-var;[0,1]}\\
	&\preceq \norm{h}^n \sqrt{\norm{R}_{\rho-var;[0,1]}}\preceq M^{\frac{n}{2}}.
\end{align*}
	The iterated integrals of $x$ are simply given by Young's integral. Hence we have
	\begin{align*}
		\sup_{\{x: I(x)<M\}}\norm{S_2(x)}_{p-var;[0,1]}\preceq M^{\frac{n}{2}}.
	\end{align*}
	By the uniform continuity of the map $S_2$ on bounded sets, it suffices to show that
	\[\sup_{\{x: I(x)<M\}}\sup_{t\in [0,1]}\abs{x_t-\Phi^l(x)_t}\rightarrow 0,  \]
	and \eqref{Uniform on bounded sets} follows from interpolation. It is easy to see that
	\begin{align*}
		\sup_{\{x: I(x)<M\}}\sup_{t\in [0,1]}\abs{x_t-\Phi^l(x)_t}\leq \sup_{\{x: I(x)<M\}}\max_{t_i \in Q^l} \norm{x }_{\rho-var;[t^l_i,t^l_{i+1}]}\leq (2M)^{\frac{n}{2}}\max_{t_i \in Q^l} R([t^l_i , t^l_{i+1}]).
	\end{align*}
	From which we deduce that 
	\[ \lim_{l\rightarrow \infty}\sup_{\{x: I(x)<M\}}\sup_{t\in [0,1]}\abs{x_t-\Phi^l(x)_t}=0.\]
	The proof is complete.
\end{proof}

\begin{theorem}
	Suppose assumption \ref{Assumption 1} holds for the components of $X_t$. The family of random processes $\{\boldsymbol{X}_t(\epsilon\omega)\}_{\epsilon>0}$ satisfies a large deviation principle on the space of geometric rough paths $G\mathscr{C}^{p-var}([0,1], \mathbb{R}^d)$ with good rate function
	\begin{equation*}
		I(\boldsymbol{x})=\begin{cases*}
			\inf\left\{\frac{1}{2}\norm{h}^2_{\mathcal{H}^d}\mid\ h\in\mathcal{H}^d,\ \langle f_t, h_i^{\otimes n} \rangle=\pi_1(\boldsymbol{x})^i,\ 1\leq i\leq d \right\} ,\\
			+\infty\;\; \text{if no such $h$ exists.}
		\end{cases*}
	\end{equation*}
	Here $\pi_1$ is the projection map defined as $\pi_1(\boldsymbol{Y})=Y$ for $\boldsymbol{Y}=(Y, \mathbb{Y})\in G\mathscr{C}^{p-var}([0,1], \mathbb{R}^d)$.
\end{theorem}
\begin{proof}
	With proposition \ref{LDP for level 1} at hand, by theorem 4.2.23 of \cite{dembo2009large} we only need to show the exponential goodness of piece-wise  approximations and a
	(uniform) continuity property on level sets of the good rate function. But these
	properties are the exact contents of lemma \ref{Exponential goodness} and lemma \ref{Continuity property on level sets}.  
\end{proof}

\section{Application to rough differential equations}
The rough lifts we constructed in section 3 allow us to study RDEs driven by $X_t$. 
In this section, let $\{V_i\}_{1\leq i\leq d}\subset C^\infty_b(\mathbb{R}^d)$; smooth functions bounded together with all derivatives. We consider
\begin{equation}
	\label{RDE}
	Y_t=y_0+\sum_{i=1}^{d}\int_{0}^{t}V_i(Y_s)d\boldsymbol{X}^i_s,\ y_0\in\mathbb{R}^d,\ t\in[0,1].
\end{equation}
All our results can be easily generalized to include a drift.
\subsection{Malliavin derivatives of $Y_t$}\hfill\\
 Let us digress a little and recall the Fa\`a di Bruno's formula for high order derivatives of composite functions, which will be helpful later.
 \begin{lemma}
 	\label{Faa di Bruno}
 	Let $f,g\in C^\infty(\mathbb{R})$, then for any $k\geq 1$
 	\[  \frac{d^k}{dx^k}f(g(x))=\sum_{\pi\in \Pi_{k} } f^{(\abs{\pi})}(g(x))\prod_{\alpha\in\pi}g^{(\abs{\alpha})}(x).  \]
 	Here $\Pi_{k}$ is the set of all partitions of the set $\{1,2, \cdots, k \}$, $\alpha\in\pi$ means $\alpha$ runs through all the blocks of the partition $\pi$ and $\abs{\pi}$ denotes the cardinality of $\pi$.
 \end{lemma}
	Take $f^{(4)}(g(x))$ as an example. We can take $\pi=\{ \{1,3\},\ \{2,4\} \}$, which has $\abs{\pi}=2$. For this partition $\pi$, it has two blocks $\{1,2\}$ and $\{3,4 \}$. Hence
	\[ \prod_{\alpha\in\pi}g^{(\abs{\alpha})}(x)=g^{(\abs{\{1,3\}})}(x)g^{(\abs{\{2,4\}})}(x)=(g''(x))^2.  \]
	Multiplying these two terms gives $f''(g(x))(g''(x))^2$, which is the term in $f^{(4)}(g(x))$ corresponding to the partition $\pi=\{ \{1,3\},\ \{2,4\} \}$.

Coming back to the RDE \eqref{RDE}, we make an important observation that the Malliavin derivatives $D^kY_t$ can be regarded as solutions to RDEs driven by $\boldsymbol{\hat{X}}_t$ along linear vector fields. This is the content of the next
\begin{proposition}
	Suppose that assumption \ref{Assumption 2} holds for the components of $X_t$. Let $Y_t$ be given by \eqref{RDE}. Then, for $k\geq 1$,
	\begin{align}
		D^kY_t&=\sum_{i=1}^{d}\int_{0}^{t}DV_i(Y_s)D^kY_sd\boldsymbol{X}^i_s+\sum_{i=1}^{d}\int_{0}^{t}\sum_{\pi\in \Pi_{k}\setminus e}D^{(\abs{\pi})}V_i(Y_s) \left(\hat{\otimes}_{\alpha\in\pi}D^{(\abs{\alpha})}Y_s \right)d\boldsymbol{X}^i_s\label{D^kY_t}\\
		&+\sum_{i=1}^{d}\sum_{r=1}^{k\wedge n}\binom{n}{r}\int_{0}^{t}\sum_{\pi\in \Pi_{k-r}}D^{(\abs{\pi})}V_i(Y_s) \left(\hat{\otimes}_{\alpha\in\pi}D^{(\abs{\alpha})}Y_s \right)\hat{\otimes}d\boldsymbol{D^rX}^i_s,\nonumber
	\end{align}
	where $e$ in the second term on the right-hand side is the trivial partition of $\{1, 2, \cdots, k \}$.
\end{proposition}
\begin{proof}
	Let $\{Q^l\}_{l\geq 1}$ be a sequence of increasing partitions of $[0,1]$. Consider
	\begin{equation}
		\label{Approximation RDE}
		Y^l_t=y_0+\sum_{i=1}^{d}\int_{0}^{t}V_i(Y^l_s)dX^{l;i}_s.
	\end{equation}
	We know from standard rough paths theory that $Y^l_t$ converges to $Y_t$ uniformly, almost surely. We thus infer that for any $h\in \mathcal{H}$,
	\begin{align*}
		\langle DY_t(\omega), h \rangle_{\mathcal{H}}&=\lim_{r\rightarrow 0}\frac{Y_t(\omega+rh)-Y_t(\omega)}{r}=\lim_{r\rightarrow 0}\lim_{l\rightarrow \infty}\frac{Y^l_t(\omega+rh)-Y^l_t(\omega)}{r}\\
		&=\lim_{l\rightarrow \infty} \langle DY^l_t(\omega), h \rangle_{\mathcal{H}}.
	\end{align*}
	As a result, $DY_t=\lim_{l\rightarrow \infty} DY^l_t$. Apply Malliavin derivative to both sides of \eqref{Approximation RDE}, we get
	\begin{align*}
		DY^l_t=\sum_{i=1}^{d}\int_{0}^{t}DV_i(Y^l_s)DY^l_sdX^{l;i}_s+\sum_{i=1}^{d}\int_{0}^{t}V_i(Y^l_s)dDX^{l;i}_s.
	\end{align*}
	Thanks to the lift above $\hat{X}_t$ that we constructed in theorem \ref{Enhance piece-wise linear rough lift}, the previous ODE, by definition, converges to
	\begin{align*}
		DY_t=\sum_{i=1}^{d}\int_{0}^{t}DV_i(Y_s)DY_sd\boldsymbol{X}^i_s+\sum_{i=1}^{d}\int_{0}^{t}V_i(Y_s)d\boldsymbol{DX}^i_s.
	\end{align*}
	This settles the case for $k=1$. 
	
	One can iterate this processes, to see
	\begin{align*}
		D^kY_t=\sum_{i=1}^{d}\int_{0}^{t}D^k\left(V_i(Y_s)d\boldsymbol{X}^i_s\right)=\sum_{i=1}^{d}\sum_{r=0}^{k\wedge n}\binom{n}{r}\int_{0}^{t}D^{k-r}V_i(Y_s)\hat{\otimes} d\boldsymbol{D^rX}^i_s.
	\end{align*}
	Finally, we apply Fa\`a di Bruno's formula (lemma \ref{Faa di Bruno}) for all the composite derivative terms $D^{k-r}V_i(Y_s)$ with the usual multiplication replaced by symmetry tensor product $\hat{\otimes}$. Equation \eqref{D^kY_t} then follows. The proof is complete.
\end{proof}

\begin{remark}
	We emphasis that we intentionally wrote \eqref{D^kY_t} in its current from (singling out the first term which is corresponding to the trivial partition $e$) so that $D^kY_t$ only shows up in the first term on the right-hand side. It then becomes apparent that $D^kY_t$ solves a RDE driven by $\boldsymbol{\hat{X}}_t$ with linear vector fields. The combinatorial nature of \eqref{D^kY_t} is not relevant to our purpose. 
\end{remark}
\begin{proposition}
	\label{Linear RDE estimate}
	Let $S, H$ be Banach spaces and $\boldsymbol{X}_t$ a rough path of $p$-variation on $H$ for some $p\geq 1$. Assume that $V=\{V_{i}\}_{1\leq i\leq d}$ is a collection of linear vector fields on $S$, defined by
	\begin{equation}
		\label{Linear vector field}
		V^i_{t}(z)=A^i_{t}z+B^i_{t},
	\end{equation}
	where $A^i_{t}\in L(S, L(H, S))$ and $B^i_{t}\in L(H, S)$ for all $1\leq i\leq d$. Let $Y_t$ be the solution to the following RDE
	\[ Y_t=y_0+\sum_{i=1}^{d}\int_{0}^{t}V^i_t(Y_s)d\boldsymbol{X}^i_s,\ y_0\in S. \]
	Then, for any $\alpha>0$, we can find a constant $C>0$ that only depends on $p,\alpha$, such that
\begin{align*}
	\norm{Y_t}_{\infty}\leq C\left(\abs{y_0}+\norm{B_t}_\infty\norm{\boldsymbol{X}_t}_{p-var;[0,1]}\right)\exp\left(C\norm{A_t}_\infty \max(1, \alpha^{-1} )(2N_\alpha(\boldsymbol{\hat{X}})+1)\right).
\end{align*}
\end{proposition}
\begin{proof}
	We only need to make two adjustments to the arguments of theorem 10.53 of \cite{friz2010multidimensional}. First, we separate the roles of $A^i$ and $B^i$ instead of introducing a uniform upper bound. Second, just as lemma 4.5 of \cite{MR3112937}, we replace each use lemma 10.63 by Remark 10.64 from \cite{friz2010multidimensional}. The whole procedure is long but elementary. What emerges from it is given by
	\begin{align*}
		\norm{Y_t}_{\infty}\leq C(\norm{y_0}+\norm{B_r}_{\infty}\norm{\boldsymbol{X}}_{p-var;[s,t]})\exp\left(C\norm{A_r}_{\infty}\max(1, \alpha^{-1} )M_{\alpha}(\boldsymbol{\hat{X}}) \right)
	\end{align*}
	Finally, our conclusion follows from the relation $M_{\alpha}(\boldsymbol{\hat{X}})\leq 2N_\alpha(\boldsymbol{\hat{X}})+1 $.
\end{proof}
\begin{remark}
	As we have already seen in \eqref{D^kY_t}, the RDE of $D^kY_t$ has polynomials of all lower order Malliavin derivatives of $Y_t$ in the place of $B^i_t$ in \eqref{Linear vector field}. It is hopeless to expect $\norm{B_t}_{\infty}$ to be exponentially integrable. Thus, to study the integrability of $D^kY_t$, we have to separate the roles of $A^i$ and $B^i$ so that $\norm{B_t}_{\infty}$ is outside the exponential. 
\end{remark}

We are ready for our main result.
\begin{theorem}
	\label{Malliavin derivative integrability}
	Suppose that assumption \ref{Assumption 2} holds for the components of $X_t$. Let $Y_t$ be the solution to the following RDE
	 \begin{equation*}
	 	Y_t=y_0+\sum_{i=1}^{d}\int_{0}^{t}V_i(Y_s)d\boldsymbol{X}^i_s,\ y_0\in\mathbb{R}^d,
	 \end{equation*}
	where $\{V_i\}_{1\leq i\leq d}\subset C^\infty_b$. For any $\alpha>0$, if we can find a constant $\eta>0$, such that
	\[ \mathbb{E}\exp\left(\eta N_\alpha(\boldsymbol{\hat{X}})\right)<+\infty.  \]
	Then, for any $t\in[0,1]$ and $k\geq 1$, we have $\norm{D^kY_t}_\infty\in L^{\xi/k}(\Omega)$, where $\xi$ is any positive number such that
	\[ \xi<\psi:=\frac{\eta}{2C_{\ref{Linear RDE estimate}}\norm{V}_{C^\infty_b} \max(1, \alpha^{-1} )},  \]
	and $C_{\ref{Linear RDE estimate}}$ is the constant in proposition \ref{Linear RDE estimate}. 
\end{theorem}

\begin{proof}
	We apply proposition \ref{Linear RDE estimate} to \eqref{D^kY_t} and get
	\begin{align}
		\label{For D^kY_t}
		\norm{D^kY_t}_\infty\leq
		 C_{\ref{Linear RDE estimate}}\cdot\Gamma(k)\cdot\norm{\boldsymbol{\hat{X}}_t}_{p-var;[0,1]}\cdot\exp(\Xi),
	\end{align}
	 where
	 \begin{align}
	 	\label{Gamma(k)}
	 	\Gamma(k)=&\norm{V}_{C^\infty_b}\sum_{i=1}^{d}\sum_{\pi\in \Pi_{k}\setminus e} \norm{\hat{\otimes}_{\alpha\in\pi}D^{(\abs{\alpha})}Y_s}_{\infty}\\
	 	&+\norm{V}_{C^\infty_b}\sum_{i=1}^{d}\sum_{r=1}^{k\wedge n}\binom{n}{r}\sum_{\pi\in \Pi_{k-r}} \norm{\hat{\otimes}_{\alpha\in\pi}D^{(\abs{\alpha})}Y_s}_{\infty},\nonumber
	 \end{align}
 	and
 	\[\exp(\Xi)= \exp\left(C_{\ref{Linear RDE estimate}}\norm{V}_{C^\infty_b} \max(1, \alpha^{-1} )(2N_\alpha(\boldsymbol{\hat{X}})+1)\right). \]
 	We claim that there exist positive $a(k), b(k)\geq 0$ such that
 	\begin{equation}
 		\label{Claim}
 		\Gamma(k)\leq \left( \norm{V}_{C^\infty_b}\vee 1 \right)^{a(k)}\left(\norm{\boldsymbol{\hat{X}}_t}_{p-var;[0,1]}  \vee 1\right)^{b(k)}\exp\left((k-1)\Xi\right).
 	\end{equation}
 	We prove it by induction. By setting $k=1$, \eqref{D^kY_t} gives
 	\[DY_t=\sum_{i=1}^{d}\int_{0}^{t}DV_i(Y_s)DY_sd\boldsymbol{X}^i_s+\sum_{i=1}^{d}\int_{0}^{t}V_i(Y_s)d\boldsymbol{DX}^i_s. \] 
 	Thus, we have 
 	\[ \Gamma(1)=\norm{V}_{C^\infty_b}, \]
 	which is our claim \eqref{Claim} for $k=1$ with $a(1)=1, b(1)=0$.	Now suppose \eqref{Claim} is true for $1\leq m\leq k-1$. Then, let us observe that each tensor product term in \eqref{Gamma(k)} has the form
 	\[  \norm{D^{i_1}Y_r\hat{\otimes}D^{i_2}Y_r\hat{\otimes}\cdots \hat{\otimes} D^{i_\theta}Y_r}_{\infty},\ 2\leq\theta\leq k,\ i_1+i_2+\cdots+i_\theta\leq k.  \]
 	Apply proposition \ref{Linear RDE estimate} to each Malliavin derivative above gives
 	\[\norm{D^{i_1}Y_r\hat{\otimes}D^{i_2}Y_r\hat{\otimes}\cdots \hat{\otimes} D^{i_\theta}Y_r}_{\infty}\preceq \norm{\boldsymbol{\hat{X}}_t}^\theta_{p-var;[0,1]} \Gamma(i_1)\Gamma(i_2)\cdots \Gamma(i_\theta) \exp(\theta\Xi). \]
 	We apply the induction hypothesis and get
 	\begin{align*}
 		\Gamma(k)&\preceq \norm{V}^{1+a(i_1)+\cdots +a(i_\theta)}_{C^\infty_b}\norm{\boldsymbol{\hat{X}}_t}_{p-var;[0,1]}^{\theta+\cdots+b(i_\theta)}\exp\left((i_1+i_2+\cdots+i_\theta-\theta+1)\Xi\right)\\
 		&\preceq \left( \norm{V}_{C^\infty_b}\vee 1 \right)^{a(k)}\left(  \norm{\boldsymbol{\hat{X}}_t}_{p-var;[0,1]} \vee 1\right)^{b(k)}\exp((k-1)\Xi),\nonumber
 	\end{align*}
 for some appropriate numbers $a(k), b(k)$.	This finishes the proof for our claim \eqref{Claim}. 
 	
 	From \eqref{Claim} and \eqref{For D^kY_t}, we infer that
 	\[\norm{D^kY_t}_\infty\preceq
 	C_{\ref{Linear RDE estimate}}  \left( \norm{V}_{C^\infty_b}\vee 1 \right)^{a(k)}\left(  \norm{\boldsymbol{\hat{X}}_t}_{p-var;[0,1]} \vee 1\right)^{b(k)+1}\exp(k\Xi).  \]
 	Since $\norm{V}_{C^\infty_b}$ is just a fixed constant, we have by definitions of $\eta ,\xi$ and $\Xi$ that
 	\[ \mathbb{E}\norm{D^kY_t}^{\xi/k}_\infty\preceq \left(\mathbb{E}\exp\left(\eta N_\alpha(\boldsymbol{\hat{X}})\right)\right)^{\frac{\xi}{\eta}}\left(\mathbb{E}\norm{\boldsymbol{\hat{X}}_t}^{(b(k)+1)q'\xi/k}_{p-var;[0,1]}\right)^{1/q'}<+\infty,  \]
 	where $q'$ is the H\"older conjugate of $\eta/\xi$. The proof is complete.
 	\end{proof}
 
 We can promptly recover the main result of \cite{MR3229800} now.
 \begin{corollary}[Theorem 1.2 of \cite{MR3229800}]
 	\label{Inahama main}
 	Let $n=1$. Suppose assumption \ref{Assumption 2} (or equivalently assumption \ref{Assumption 1}) holds for the components of $X_t$. Then $Y_t\in\mathbb{D}^{\infty}$.
 \end{corollary}
 \begin{proof}
 	We only need to show $D^kY_t\in L^{p}(\Omega;\mathcal{H}^{\otimes k})$ for all $p\geq 1$. By theorem 6.3 of \cite{MR3112937}, we can find constant $C>0$ such that
 	\[ \mathbb{E}\exp\{CN_{\alpha}(\boldsymbol{X})^{2/\rho}  \}<+\infty. \]
 	Since $\rho<2$, theorem \ref{Malliavin derivative integrability} applies with any $\eta>0$. The desired result follows.
 \end{proof}

 	\subsection{Jacobian processes}
 	It is well known that the Jacobian processes $J^Y_t$ of $Y_t$ solves 
 	\[J^Y_t=\sum_{i=1}^{d}\int_{0}^{t}DV_i(Y_s)J^Y_sd\boldsymbol{X}^i_s,\ J^Y_0=I_{d\times d}.  \]
 	We thus have
 	\begin{theorem}
 		Under the assumptions of theorem \ref{Malliavin derivative integrability}, we have
 		\[\abs{J^Y_t}_{\infty}\in L^{\psi}(\Omega).\]
 	\end{theorem}
 	\begin{proof}
 		This is immediate from proposition \ref{Linear RDE estimate} with $\norm{B}_{\infty}=0$. 
 	\end{proof}
 Just as we did in corollary \ref{Inahama main}, we have
 \begin{corollary}[Theorem 6.5 of \cite{MR3112937}]
 	Let $n=1$. Suppose assumption \ref{Assumption 2} (or equivalently assumption \ref{Assumption 1}) holds for the components of $X_t$. Then $J^Y_t\in L^{p}(\Omega, \mathbb{R}^{d\times d})$ for all $p\geq 1$.
 \end{corollary}
 
\subsection{General tail behavior}

We have seen that the integrabilities of $D^kY_t$ and $J^Y_t$ are determined by the exponential integrability of $N_\alpha(\boldsymbol{\hat{X}})$. Now we apply the method of \cite{MR3112937} to study it. 

First, we characterize how the $p$-variation of $\boldsymbol{\hat{X}}_t$ changes under $\mathcal{H}$ translation.
\begin{proposition}
	\label{Translation behavior}
	Suppose that assumption \ref{Assumption 2} holds for the components of $X_t$. Then, for any $p>2\rho$ and $h\in \mathcal{H}^{ d}$, we have
	\begin{equation*}
		\norm{\boldsymbol{T}_h\hat{\boldsymbol{X}}_t}^p_{p-var;[0,1]}\leq C(n,p)\norm{\hat{\boldsymbol{X}}}^p_{p-var;[0,1]}\left(\frac{\norm{h}^{np+p/2}_{\mathcal{H}^d}-1}{\norm{h}^{p/2}_{\mathcal{H}^d}-1} \vee \frac{\norm{h}^{(n+1)p}_{\mathcal{H}^d}-1}{\norm{h}^{p}_{\mathcal{H}^d}-1} \right).
	\end{equation*}
\end{proposition}

\begin{proof}
	Without loss of generality, we assume that $d=1$. For level 1, we have for each $0\leq k\leq n$ that
	\begin{align*}
		D^kX_t(\omega+h)=D^kX_t(\omega)+\sum_{j=1}^{n-k}\frac{1}{j!}\cdot \langle D^{k+j}X_t(\omega), h^{\otimes j} \rangle_{\mathcal{H}^{\otimes j}}.
	\end{align*}
	As a result,
	\begin{align}
		&\norm{D^kX_t(\omega+h)}^p_{p-var;[0,1]}\nonumber\\
		&\leq C(n,p,k) \sum_{j=0}^{n-k}\norm{D^{k+j}X_t}^p_{p-var;[0,1]}\norm{h}^{pj}_{\mathcal{H}} \nonumber\\
		&\leq C(n,p,k)\norm{\hat{\boldsymbol{X}}}^p_{p-var;[0,1]}\sum_{j=0}^{n-k}\norm{h}^{jp}_{\mathcal{H}}\nonumber\\
		&\leq C(n,p)\norm{\hat{\boldsymbol{X}}}^p_{p-var;[0,1]}\frac{\norm{h}^{(n+1)p}_{\mathcal{H}}-1}{\norm{h}^{p}_{\mathcal{H}}-1} \label{Translated level 1}
	\end{align}
	Similarly for level 2, we have for any $0\leq k,j\leq n$ that
	\begin{align*}
		&\int_{s}^{t}D^kX_{s,r}(\omega+h)\otimes dD^jX_r(\omega+h)\\
		&=\sum_{\substack{0\leq i_1\leq n-k\\ 0\leq i_2\leq n-j}}C(n,i_1,i_2)\int_{s}^{t}\langle D^{k+i_1}X_{s,r}(\omega), h^{\otimes i_1} \rangle_{\mathcal{H}^{\otimes i_1}}\otimes d\langle D^{j+i_2}X_r(\omega), h^{\otimes i_2} \rangle_{\mathcal{H}^{\otimes i_2}}.
	\end{align*}
	Therefore,
	\begin{align*}
		&\norm{	\int_{s}^{t}D^kX_{s,r}(\omega+h)\otimes dD^jX_r(\omega+h)}^{p/2}_{p/2-var;[0,1]}\\
		&\leq C(n,p,k,j)\sum_{\substack{0\leq i_1\leq n-k\\ 0\leq i_2\leq n-j} }\norm{\int_{s}^{t}D^kX_{s,r}(\omega)\otimes dD^jX_r(\omega)}^{p/2}_{p/2-var;[0,1]}\norm{h}_{\mathcal{H}}^{(i_1+i_2)p/2}.
	\end{align*}
	 As a result,
\begin{align}
	&\norm{	\int_{s}^{t}D^kX_{s,r}(\omega+h)\otimes dD^jX_r(\omega+h)}^{p/2}_{p/2-var;[0,1]}\nonumber\\
	&\leq C(n,p)\norm{\hat{\boldsymbol{X}}}^p_{p-var;[0,1]}\frac{\norm{h}^{np+p/2}_{\mathcal{H}}-1}{\norm{h}^{p/2}_{\mathcal{H}}-1}\label{Translated level 2}
\end{align}
	Combining \eqref{Translated level 1} \eqref{Translated level 2}, we conclude
	\begin{equation*}
		\norm{\boldsymbol{T}_h\hat{\boldsymbol{X}}_t}^p_{p-var;[0,1]}\leq C(n,p)\norm{\hat{\boldsymbol{X}}}^p_{p-var;[0,1]}\left(\frac{\norm{h}^{np+p/2}_{\mathcal{H}}-1}{\norm{h}^{p/2}_{\mathcal{H}}-1} \vee \frac{\norm{h}^{(n+1)p}_{\mathcal{H}}-1}{\norm{h}^{p}_{\mathcal{H}}-1} \right).
	\end{equation*}
\end{proof}

Next we establish the tail behavior of $N_\alpha(\boldsymbol{\hat{X}})$.
\begin{proposition}
	\label{Tail}
	Suppose that assumption \ref{Assumption 2} holds for the components of $X_t$. For any $p>2\rho$, we can find $\kappa>0$ such that, when $M>0$ is big enough,
	\begin{equation}
		\label{Tail estimate}
		\mathbb{P}\left\{ N_\alpha(\boldsymbol{\hat{X}})>M  \right\}\leq\frac{1}{2}\exp\left\{  -\frac{1}{2}\left( \frac{\alpha M}{4C(n,p)\kappa}\right)^{2/(np)} \right\},
	\end{equation}	
	where $C(n,p)$ is the constant appears in proposition \ref{Translation behavior}.
\end{proposition}
\begin{proof}
	Without loss of generality, we assume $d=1$. For any $\kappa>0$, let 
	\[ F(\kappa):=\left\{ \omega\in\Omega : \norm{\hat{\boldsymbol{X}}_t}^p_{p-var;[0,1]}< \kappa     \right\}. \]
	We claim that for $\kappa>0$ fixed and $M>0$ big enough, we have
	\begin{equation}
		\label{Isoperimetric setting}
		\left\{ N_\alpha(\boldsymbol{\hat{X}})>M  \right\}\subset  \Omega\setminus (F(\kappa)+r_MB(\mathcal{H}) ),
	\end{equation}
	where $B(\mathcal{H})$ is the unit ball in $\mathcal{H}$,
	\[r_M=\left( \frac{\alpha M}{4C(n,p)\kappa}  \right)^{1/(np)}, \]
	and $C(n,p)$ is the constant appears in proposition \ref{Translation behavior}. 
	
	To see this, observe that for any $\omega\in F(\kappa)+r_MB(\mathcal{H})$, we can write
	\[\omega=\omega'+rh,\ \omega'\in F(\kappa),\ 0\leq r\leq r_M,\ h\in B(\mathcal{H}).     \]
	Then, by \eqref{Boung on N} and proposition \ref{Translation behavior}
	\begin{align*}
		\alpha N_\alpha(\boldsymbol{\hat{X}})(\omega)&\leq \norm{\hat{\boldsymbol{X}}_t(\omega)}^p_{p-var;[0,1]}=\norm{\boldsymbol{T}_{rh}\hat{\boldsymbol{X}}_t(\omega')}^p_{p-var;[0,1]}\\
		&\leq C(n,p)\norm{\hat{\boldsymbol{X}}_t(\omega')}^p_{p-var;[0,1]}\left(\frac{\norm{rh}^{np+p/2}_{\mathcal{H}}-1}{\norm{rh}^{p/2}_{\mathcal{H}}-1} \vee \frac{\norm{rh}^{(n+1)p}_{\mathcal{H}}-1}{\norm{rh}^{p}_{\mathcal{H}}-1}\right)\\
		&\leq C(n,p)\kappa\left( \frac{(r_M)^{np+p/2}-1}{(r_M)^{p/2}-1} \vee \frac{(r_M)^{(n+1)p}-1}{(r_M)^{p}-1} \right).
	\end{align*}
	If we can find 
	\[ \omega\in \left\{ N_\alpha(\boldsymbol{\hat{X}})>M  \right\}\cap \left(F(\kappa)+r_MB(\mathcal{H})\right), \]
	then
		\begin{equation}
		\label{Bound on r_M}
		\frac{\alpha M}{C(n,p)\kappa}\leq  \frac{(r_M)^{np+p/2}-1}{(r_M)^{p/2}-1} \vee \frac{(r_M)^{(n+1)p}-1}{(r_M)^{p}-1}.
	\end{equation}
	When $M$ is large enough, \eqref{Bound on r_M} implies
	\begin{equation*}
		\frac{\alpha M}{2C(n,p)\kappa}\leq (r_M)^{np}=\frac{\alpha M}{4C(n,p)\kappa},
	\end{equation*}
	which is a contradiction. Thus, we must have \eqref{Isoperimetric setting}.
	
	Finally, we may choose $\kappa$ big enough so that $\mathbb{P}(F(\kappa))\geq 1/2$. Then, by Gaussian isoperimetric inequality, we have
	\begin{align*}
		\mathbb{P}\left\{ N_\alpha>M  \right\}&\leq 1-\mathbb{P}(F(\kappa)+r_MB(\mathcal{H}) )\\
		&\leq \frac{1}{2}\exp\left( -r^2_M/2 \right)=\frac{1}{2}\exp\left\{  -\frac{1}{2}\left( \frac{\alpha M}{4C(n,p)\kappa}\right)^{2/(np)} \right\}.
	\end{align*}
	The proof is complete.
\end{proof}
	
	\begin{remark}
		It is worthwhile making a number of comments. When substituting $n=1$ into \eqref{Tail estimate}, our result is weaker than theorem 6.3 of \cite{MR3112937}. To see where the sharpness loss is coming from, we need to go back to the $\mathcal{H}$ translation of $X_t$
		\begin{equation*}
			X_t(\omega+rh)-X_t(\omega)=\sum_{k=1}^{n}\frac{r^k}{k!}\cdot\langle D^kX_t(\omega), h^{\otimes k} \rangle_{\mathcal{H}^{\otimes k}},\ r>0,\ h\in\mathcal{H}.
		\end{equation*}
		The term on the right-hand side can be recast in another way
		\[ \langle D^kX_t(\omega), h^{\otimes k} \rangle_{\mathcal{H}^{\otimes k}}=C(k)\langle I_{n-k}(f_t), h^{\otimes k} \rangle_{\mathcal{H}^{\otimes k}}=C(k)I_{n-k}\left(\langle f_t, h^{\otimes k} \rangle_{\mathcal{H}^{\otimes k}}\right).\]
		The last term $\langle f_t, h^{\otimes n} \rangle_{\mathcal{H}^{\otimes n}}$, corresponding to $k=n$, is deterministic with finite $\rho$-variation (proposition \ref{Embedding}), while others are random with finite $p$-variation (or $2\rho^+$-variation). This is due to the fact that multiple Wiener integrals tend to make the trajectories ``twice more irregular'' than its kernels. For instance, in the classical Wiener space, Brownian motion can be written as $W_t=I_1(f_t)$, with $f_t=1_{[0,t]}$, which has finite $1$-variation in $L^2([0,1])$. While the trajectories of $W_t$ have finite $2^+$-variation.
		
		Coming back to our discussion. When $n=1$ (the Gaussian case), $\mathcal{H}$ translation leads to adding a deterministic process $\langle f_t, h \rangle_{\mathcal{H}}$ whose regularity is significantly better than that of $X_t$, which in turn leads to a slower growth of $\norm{X_t(\omega+h)}_{p-var;[0,1]}$ and $N_\alpha(\hat{\boldsymbol{X}}(\omega+h))$. This is the essential reason, estimate in theorem 6.3 of \cite{MR3112937} can replace the exponent $2/p$ in \eqref{Tail estimate} by $2/\rho$. When $n>1$, the adding terms must include random processes of form $I_{n-k}\left(\langle f_t, h^{\otimes k} \rangle_{\mathcal{H}^{\otimes k}}\right)$ that have same regularity as $X_t$. As a result,  $\norm{X_t(\omega+h)}_{p-var;[0,1]}$ and $N_\alpha(\hat{\boldsymbol{X}}(\omega+h))$ grow drastically faster than the $n=1$ case. With current method and no additional assumption on $\{f_t\}_{0\leq t\leq 1}$, estimate in \eqref{Tail estimate} is the best we can get for general chaos processes. 
		
		It will be interesting to investigate what type of additional algebraic structures we can put on the kernels $\{f_t\}_{0\leq t\leq 1}$ to make up for the loss of Gaussianity so that \eqref{Tail estimate} can be improved.
	\end{remark}
	Although proposition \ref{Tail} does not provide exponential integrability for $N_\alpha(\hat{\boldsymbol{X}})$, we can still achieve local integrability, which is sufficient for many studies including the absolute continuity of the law of $Y_t$.  
\begin{corollary}
	Under the assumptions of theorem \ref{Malliavin derivative integrability}, we have that for any $t\in[0,1]$, $Y_t\in \mathbb{D}^{\infty}_{loc}$.
\end{corollary}
\begin{proof}
	From proposition \ref{Tail}, we know that for any $\alpha>0$, $\mathbb{E}(N_\alpha(\hat{X}))^p<\infty$ for all $p>0$. As a result,
	\[ \cup_{M\in \mathbb{N}^+}\{\omega\in\Omega : N_\alpha(\hat{X}(\omega))\leq M  \}=\Omega. \]
	From theorem \ref{Malliavin derivative integrability}, we know that $D^kY_t\cdot \boldsymbol{1}_{ \{N_\alpha(\hat{X})\leq M\} }$ is bounded almost surely, hence $L^p(\Omega)$ integrable for any $p>1$. This gives the desired localization of $D^kY_t$ for any $k\geq 1$.
\end{proof}

\bibliographystyle{plain}
\bibliography{reference}

\end{document}